\documentclass{amsart}
\usepackage{amstext,amssymb,amsthm,amsopn,newlfont,graphpap,graphics,graphicx,mathrsfs,enumitem}\usepackage{amstext,amssymb,amsthm,amsopn,newlfont,graphpap,graphics,graphicx,mathrsfs,enumitem,tikz-cd}
\usepackage[parfill]{parskip}
\usepackage[noadjust]{cite}
\usepackage{epigraph}
\usepackage[colorlinks=true,
            linkcolor=red,
            urlcolor=blue,
            citecolor=magenta]{hyperref}
\usepackage{color}
\usepackage{mathrsfs}
\allowdisplaybreaks
\theoremstyle{plain}
\newtheorem{thm}{Theorem}[section]
\newtheorem*{thm*}{Theorem}
\newtheorem{prop}{Proposition}[section]
\newtheorem*{prop*}{Proposition}
\newtheorem{cor}{Corollary}[section]
\newtheorem*{cor*}{Corollary}
\newtheorem{lem}{Lemma}[section]
\newtheorem*{lem*}{Lemma}
\theoremstyle{definition}
\newtheorem{defn}{Definition}[section]
\newtheorem*{defn*}{Definition}
\newtheorem{exmp}{Example}[section]
\newtheorem*{exmp*}{Example}
\newtheorem{exmps}[exmp]{Examples}
\newtheorem*{exmps*}{Examples}

\newtheorem{rem}{Remark}[section]
\newtheorem*{rem*}{Remark}
\newtheorem{rems}{Remarks}[section]
\newtheorem*{rems*}{Remarks}

\newtheorem*{note*}{Note}
\newcommand{\N}{{\mathbb N}}
\newcommand{\Z}{{\mathbb Z}}
\newcommand{\R}{{\mathbb R}}
\newcommand{\C}{{\mathbb C}}
\newcommand{\F}{{\mathbb F}}

\newcommand{\eps}{\varepsilon}
\renewcommand{\iff}{\: \Leftrightarrow\: }
\renewcommand{\bar}{\overline}

\numberwithin{equation}{section}

\DeclareMathOperator{\spa}{span}
\DeclareMathOperator{\orb}{orb}
 
\DeclareMathOperator{\Per}{Per} 
\begin{document}
\title[Hypercyclicity and linear chaos in a nonclassical sequence space]
{On hypercyclicity and linear chaos\\
in a nonclassical sequence space and beyond}
\author[Marat V. Markin]{Marat V. Markin}
\address{
Department of Mathematics\newline
California State University, Fresno\newline
5245 N. Backer Avenue, M/S PB 108\newline
Fresno, CA 93740-8001, USA
}
\email[corresponding author]{mmarkin@csufresno.edu}
\author{Eric Montoya}
\email{emontoya@mail.fresnostate.edu}
\subjclass{Primary 47A16, 47B37; Secondary 47A10}
\keywords{Hypercyclic vector, periodic point, hypercyclic operator, chaotic operator, spectrum}
\begin{abstract}
We analyze the hypercyclicity, chaoticity, and spectral structure of (bounded and unbounded) weighted backward shifts in a nonclassical sequence space, which the space $l_1$ of summable sequences is both isometrically isomorphic to and continuously and densely embedded into.

Based on the weighted backward shifts, we further construct new bounded and unbounded linear hypercyclic and chaotic operators both in the nonclassical sequence space and the classical space $l_1$, including those that are hypercyclic but not chaotic.
\end{abstract}
\maketitle

\section[Introduction]{Introduction}

We analyze the hypercyclicity, chaoticity, and spectral structure of the weighted backward shifts, \textit{bounded} 
\[
X\ni x:=\left(x_k\right)_{k\in \N}\mapsto A_wx:=w\left(x_{k+1}\right)_{k\in \N}\in X\quad (w\in \F)
\]
($\F:=\R$ or $\F:=\C$), as well as \textit{unbounded} 
\[
A_wx:=\left(w^kx_{k+1}\right)_{k\in \N} \quad (w\in \F,\ |w|>1)
\]
with maximal domain
\[
D(A_w):=\left\{ x:=\left(x_k\right)_{k\in \N} \in X \,\middle|\, \left(w^kx_{k+1}\right)_{k\in \N}\in X \right\},
\]
(for the first mention in the classical setting, see \cite{Rolewicz,arXiv:1811.06640}) in a nonclassical sequence space $X$ introduced in \cite{Grosse-Erdmann2000} (see also \cite{Grosse-Erdmann-Manguillot}), which the space $l_1$ of summable sequences is both isometrically isomorphic to and continuously and densely embedded into.

Based on the weighted backward shifts, we further construct new bounded and unbounded linear hypercyclic and chaotic operators both in the nonclassical sequence space $X$ and the classical space $l_1$, including those that are hypercyclic but not chaotic.

\section[Preliminaries]{Preliminaries}

The subsequent preliminaries are essential for our discourse.

\subsection{Certain Facts on Classical Sequence Spaces}\label{seccfcss}

\begin{defn}[Schauder Basis]\ \\
For a Banach space $(X,\|\cdot\|)$ over $\F$, a subset $\{e_n\}_{n\in \N}\subseteq X $ is called a \emph{Schauder basis} if
\[
\forall\, x\in X\ \exists!\, (c_k)_{k\in \N}\in \F^\N:\ x=\sum_{k=1}^\infty c_ke_k.
\]

The series is called the \emph{Schauder expansion} of $x$ and the numbers $c_k\in \F$, $k\in \N$, are called the \emph{coordinates} of $x$ relative to $\{e_n\}_{n\in \N}$ (see, e.g., \cite{Markin2018EFA,Markin2020EOT,MarkSogh2021}).
\end{defn}

Classical examples of such spaces are the spaces
\[
l_p:=\left\{ x:=\left(x_k\right)_{k\in \N} \in \F^\N\,\middle|\, \sum_{k=1}^\infty {|x_k|}^p<\infty \right\}
\]
of $p$-\textit{summable sequences} ($1\le p<\infty$) and the space 
\[
c_0:=\left\{ x:=\left(x_k\right)_{k\in \N} \in \F^\N\,\middle|\, \lim_{k\to\infty} x_k=0 \right\}
\]
of \textit{vanishing sequences}. 

For these spaces, the set  $\left\{e_n:=\left(\delta_{nk}\right)_{k\in \N}\right\}_{n\in \N}$, where 
\[
\delta_{nk}:=
\begin{cases}
1&\text{if}\ k=n,\\
0&\text{if}\ k\neq n,
\end{cases}
\]
is the \textit{Kronecker delta}, is the \textit{standard} Schauder basis and
\[
\forall\, x:=\left(x_k\right)_{k\in \N}\in X:\ x=\sum_{k=1}^\infty x_ke_k.
\]
($X:=l_p$ ($1\le p<\infty$) or $X:=c_0$).

For the space
\[
c:=\left\{ x:=\left(x_k\right)_{k\in \N} \in \F^\N\,\middle|\, \exists\,\lim_{k\to\infty} x_k\in \F \right\}
\]
of \textit{convergent sequences}, the \textit{standard} Schauder basis is
 $\left\{e_n:=\left(\delta_{nk}\right)_{k\in \N}\right\}_{n\in \Z_+}$, where
\[
e_0:=\left(1,1,1,\dots\right)\quad \text{and}\quad e_n:=\left(\delta_{nk}\right)_{k\in \N},\ n\in \N,
\]
and
\[
\forall\, x:=\left(x_k\right)_{k\in \N}\in c:\ x=\sum_{k=0}^\infty c_ke_k,
\]
where
\[
c_0=\lim_{m\to\infty}x_m,\ c_k=x_k-\lim_{m\to\infty}x_m,\ k\in \N.
\]

The space 
\[
l_\infty:=\left\{ x:=\left(x_k\right)_{k\in \N} \in \F^\N\,\middle|\, \sup_{k\in \N}|x_k|<\infty \right\}
\]
of \textit{bounded sequences} has no Schauder basis since it is not separable.

See, e.g. \cite{Markin2018EFA,Markin2020EOT,MarkSogh2021}.

\begin{rem}\label{c0clinf}
Observe that $c$ is a proper subspace of $l_\infty$ and $c_0$ is a proper subspace
(a \textit{hyperplane}) of $c$, i.e.,
\[
c_0\subset c\subset l_\infty,
\]
(see, e.g., \cite{Markin2018EFA,Markin2020EOT,arXiv:2203.02032}).
\end{rem}

Henceforth, make use of the following

\begin{thm}[General Characterization of Convergence {\cite[Theorem $1$]{MarkSogh2021}}]\label{GCC}\ \\
Let $(X,\|\cdot\|) $ be a Banach space with a Schauder basis $\left\{e_n\right\}_{n\in \N} $ and corresponding coordinate functionals $ c_n(\cdot)$, $n\in \N$.

For a sequence $(x_n)_{n\in \N}$ and a vector $x$ in $X$, 
\[
x_n\to x, \ n\to \infty,
\]
iff
\begin{enumerate}[label=(\arabic*)]
\item $\forall\, k\in \N: \ c_k(x_n)\to c_k(x)$, $n\to \infty$, and
\item $\displaystyle \forall\, \eps>0\ \exists\, K_0\in \N\ \forall\, K\ge K_0 \ \forall\, n\in \N:\ \left\|\sum_{k=K+1}^\infty c_k(x_n)e_k\right\|<\eps$.
\end{enumerate}
\end{thm}

\begin{rem}\label{remcdecss}
In the chain of the proper inclusions
\begin{equation*}
c_{00}\subset l_p\subset l_q \subset c_0,
\end{equation*}
where $1\le p<q<\infty$, (see, e.g., \cite{Markin2018EFA,Markin2020EOT}),
\begin{equation*}
l_p\hookrightarrow l_q \hookrightarrow c_0
\end{equation*}
are \textit{continuous} and \textit{dense} embeddings. 

In particular, the continuity of embeddings
\[
l_p\hookrightarrow l_q\quad (1\le p<q<\infty)
\] 
is implied by the \textit{Characterization of Convergence in $l_p$ ($1\le p<\infty$)} (see, e.g., \cite[Proposition $2.16$]{Markin2018EFA}, \cite[Proposition $2.17$]{Markin2020EOT}), which, in its turn, is a particular case of the prior general characterization \cite{MarkSogh2021}.
\end{rem}

\subsection{Spectrum}\

The spectrum $\sigma(A)$ of a closed linear operator $A$ in a complex Banach space $X$ is the union of the following pairwise disjoint sets:
\begin{equation*}
\begin{split}
& \sigma_p(A):=\left\{\lambda\in \C \,\middle|\,A-\lambda I\ \text{is \textit{not injective}, i.e., $\lambda$ is an \textit{eigenvalue} of $A$} \right\},\\
& \sigma_c(A):=\left\{\lambda\in \C \,\middle|\,A-\lambda I\ \text{is \textit{injective},
\textit{not surjective}, and $\overline{R(A-\lambda I)}=X$} \right\},\\
& \sigma_r(A):=\left\{\lambda\in \C \,\middle|\,A-\lambda I\ \text{is \textit{injective} and $\overline{R(A-\lambda I)}\neq X$} \right\}
\end{split}
\end{equation*}
($R(\cdot)$ is the \textit{range} of an operator and $\overline{\cdot}$ is the \textit{closure} of a set), called the \textit{point}, \textit{continuous}, and \textit{residual spectrum} of $A$, respectively (see, e.g., \cite{Dun-SchI,Markin2020EOT}).

\subsection{Hypercyclicity and Linear Chaos}\

\begin{defn}[Hypercyclic and Chaotic Linear Operators]\ \\ 
For a (bounded or unbounded) linear operator $A$ in a (real or complex) Banach space $X$, a nonzero vector 
\begin{equation*}
x\in C^\infty(A):=\bigcap_{n=0}^{\infty}D(A^n)
\end{equation*}
($D(\cdot)$ is the \textit{domain} of an operator, $A^0:=I$, $I$ is the \textit{identity operator} on $X$) is called \textit{hypercyclic} if its \textit{orbit} under $A$
\[
\orb(x,A):=\left\{A^nx\right\}_{n\in\Z_+}
\]
is dense in $X$, i.e.,
\[
\bar{\orb(x,A)}=X.
\]

Linear operators possessing hypercyclic vectors are said to be \textit{hypercyclic}.

If there exist an $N\in \N$ and a vector 
\[
x\in D\left(A^N\right)\quad \text{with}\quad A^Nx = x,
\]
such a vector is called a \textit{periodic point} for the operator $A$ of period $N$. If $x\ne 0$, we say that $N$ is a \textit{period} for $A$.
Hypercyclic linear operators with a dense in $X$ set $\Per(A)$ of periodic points, i.e.,
\[
\bar{\Per(A)}=X,
\]
are said to be \textit{chaotic}.
\end{defn}

See \cite{Devaney,Godefroy-Shapiro1991,B-Ch-S2001}.

\begin{exmps}\label{exmpshlc}\
\begin{enumerate}[label=\arabic*.]
    \item On the space $X:=l_p$ ($1\le p<\infty$) or $X:=c_0$, the classical Rolewicz weighted backward shifts
    $$ X \ni x:=\left(x_k\right)_{k\in \N} \mapsto A_wx:= w\left(x_{k+1}\right)_{k\in \N} \in X, $$
    where $w\in \F$ with $|w|>1$ are \textit{chaotic} \cite{Rolewicz,Godefroy-Shapiro1991}.
   \item On the nonclassical sequence space
   \[
X:=\left\{ \left(x_k\right)_{k\in \N}\in \F^\N\,\middle|\, 
\sum_{k=1}^\infty \left|\frac{x_{k+1}}{k+1}-\frac{x_{k}}{k}\right|<\infty\ \text{and}\ \lim_{k\to\infty}\frac{x_k}{k} = 0\right\},  
  \]
equipped with the norm
$$X \ni x:=\left(x_k\right)_{k\in \N}\mapsto \|x\| :=\sum_{k=1}^\infty 
\left| \frac{x_{k+1}}{k+1}-\frac{x_{k}}{k} \right|, $$
which takes the center stage in the subsequent discourse (see Section \ref{secncss}), the backward shift
    $$ X \ni x:=\left(x_k\right)_{k\in \N} \mapsto 
    Ax:= \left(x_{k+1}\right)_{k\in \N} \in X $$
    is \textit{hypercyclic} but \textit{not} densely periodic,
    and hence, \textit{not}  chaotic 
\cite{Grosse-Erdmann2000} (see also \cite[Exercise $4.1.3$]{Grosse-Erdmann-Manguillot}).
\item On an infinite-dimensional separable Banach space $(X,\|\cdot\|)$, 
the identity operator $I$ is densely periodic but \textit{not}
hypercyclic, and hence, \textit{not}  chaotic.
\end{enumerate}
\end{exmps}

\begin{samepage}
\begin{rems}\label{HCrems}\
\begin{itemize}
\item In the prior definition of hypercyclicity, the underlying space is necessarily
\textit{infinite-dimensional} and \textit{separable} (see, e.g., \cite{Grosse-Erdmann-Manguillot}).
\item For a hypercyclic linear operator $A$, the set $HC(A)$ of its hypercyclic vectors is necessarily dense in $X$, and hence, the more so, is the subspace $C^\infty(A)\supseteq HC(A)$.
\item Observe that
\[
\Per(A)=\bigcup_{N=1}^\infty \Per_N(A),
\]
where 
\[
\Per_N(A)=\ker(A^N-I),\ N\in \N
\]
is the \textit{subspace} of $N$-periodic points of $A$.
\item As immediately follows from the inclusions
\begin{equation*}
HC(A^n)\subseteq HC(A),\ \Per(A^n)\subseteq \Per(A), n\in \N,
\end{equation*}
if, for a linear operator $A$ in an infinite-dimensional separable Banach space $X$ and some $n\ge 2$, the operator $A^n$ is hypercyclic or chaotic, then $A$ is also hypercyclic or chaotic, respectively.
\end{itemize} 
\end{rems}
\end{samepage}

Prior to \cite{B-Ch-S2001,deL-E-G-E2003}, the notions of linear hypercyclicity and chaos had been studied exclusively for \textit{continuous} linear operators on Fr\'echet spaces, in particular for \textit{bounded} linear operators on Banach spaces (for a comprehensive survey, see \cite{Bayart-Matheron,Grosse-Erdmann-Manguillot}).

The following extension of \textit{Kitai's ctriterion} for bounded linear operators (see \cite{Kitai1982,Gethner-Shapiro1987}) is a useful shortcut for establishing hypercyclicity for (bounded or unbounded) linear operators without explicitly furnishing a hypercyclic vector as in \cite{Rolewicz}.

\begin{thm}[Sufficient Condition for Hypercyclicity {\cite[Theorem $2.1$]{B-Ch-S2001}}]\label{SCH}\ \\
Let $X$ be a (real or complex) infinite-dimensional separable Banach space and A be a densely defined linear operator in X such that each power $A^n$, $n \in \N$, is a closed operator. If there exists a set
\[
Y\subseteq C^\infty(A):=\bigcap_{n=1}^\infty D(A^n)
\]
dense in $X$ and a mapping $B:Y\to Y$ such that
\begin{enumerate}
\item $\forall\, x\in Y:\ ABx=x$ and
\item $\forall\, x\in  Y:\ A^n x, B^n x \to 0,\ n \to \infty,$ 
\end{enumerate}
then the operator $A$ is hypercyclic.
\end{thm}

The subsequent newly established sufficient condition for linear chaos \cite{arXiv:2106.14872}, obtained via strengthening one of the hypotheses of the prior sufficient condition for hypercyclicity, serves as a shortcut for establishing chaoticity for (bounded or unbounded) linear operators without explicitly furnishing both a hypercyclic vector and a dense set of periodic points and is fundamental for our discourse.

\begin{thm}[Sufficient Condition for Linear Chaos {\cite[Theorem $3.2$]{arXiv:2106.14872}}]\label{SCC}\ \\
Let $(X,\|\cdot\|)$ be a  (real or complex) infinite-dimensional separable Banach space and $A$ be a densely defined linear operator in $X$ such that each power $A^{n}$, $n\in\N$, is a closed operator. If there exists a set
\[
Y\subseteq C^\infty(A):=\bigcap_{n=1}^\infty D(A^n)
\]
dense in $X$ and a mapping $B:Y\to Y$ such that
\begin{enumerate}
\item $\forall\, x\in Y:\ ABx=x$ and
\item $\forall\, x\in Y\  \exists\, \alpha=\alpha(x)\in (0,1),\ \exists\, c=c(x,\alpha)>0\ \forall\, n\in \N:$
\begin{equation*}
\max\left(\|A^nx\|,\|B^nx\|\right)\le c\alpha^n,
\end{equation*}
\end{enumerate}
then the operator $A$ is chaotic.
\end{thm}

For applications, see \cite{arXiv:2106.09682}.

We also need the following statements.

\begin{cor}[Chaoticity of Powers {\cite[Corollary $4.3$]{arXiv:2106.14872}}]\label{CP}\ \\
For a chaotic linear operator $A$ in a  (real or complex) infinite-dimensional separable Banach space subject to the \textit{Sufficient Condition for Linear Chaos} (Theorem \ref{SCC}), each power $A^n$, $n\in \N$, is chaotic.
\end{cor}

\begin{thm}[Kitai {\cite[Theorem $5.6$]{Grosse-Erdmann-Manguillot}}]\label{KT}\ \\
For a bounded hypercyclic operator $A$ on a complex infinite-dimensional separable Banach space, every connected component of its spectrum $\sigma(A)$ meets the unit circle.
\end{thm}

\begin{prop}[{\cite[Proposition $5.7$]{Grosse-Erdmann-Manguillot}}]\label{P5.7}\ \\
For a bounded hypercyclic operator $A$ on a complex infinite-dimensional separable Banach space, its spectrum has no isolated points and its point spectrum $\sigma_p(A)$ contains infinitely many roots of unity.
\end{prop}

\section{Nonclassical Sequence Space}\label{secncss}

\subsection{Space and Its Elements}\


For us, underlying is the following nonclassical infinite-dimensional sequence space
   \[
X:=\left\{ \left(x_k\right)_{k\in \N}\in \F^\N\,\middle|\, 
\sum_{k=1}^\infty \left|\frac{x_{k+1}}{k+1}-\frac{x_{k}}{k}\right|<\infty\ \text{and}\ \lim_{k\to\infty}\frac{x_k}{k} = 0\right\}
  \]
introduced in \cite{Grosse-Erdmann2000}, which is separable and Banach relative to the norm
    $$X \ni x:=\left(x_k\right)_{k\in \N}\mapsto \|x\| :=\sum_{k=1}^\infty \left| \frac{x_{k+1}}{k+1}- \frac{x_{k}}{k} \right|$$
(see also \cite[Exercise $4.1.3$]{Grosse-Erdmann-Manguillot}) and referred to as $(X,\|\cdot\|)$ henceforth.

\begin{rem}\label{remse1}
Thus, a sequence $(x_k)_{k\in \mathbb{N}}\in \F^\N$
is in the space $X$ \textit{iff} it satisfies the two conditions:
\[ 
\left(\frac{x_{k+1}}{k+1}-\frac{x_{k}}{k}\right)_{k\in \mathbb{N}}\in l_1 \space\ \space\ \text{and} \space\ \space\ \left(\frac{x_k}{k}\right)_{k\in \mathbb{N}}\in c_0,
\]
the latter condition being equivalent to
\[
x_k=o(k),\ k\to\infty.
\] 

The following examples illustrate the independence of the above conditions.
\end{rem}

\begin{exmps}\label{exmpsse1}\
\begin{enumerate}[label=\arabic*.]
\item $\left(1,1,1,\dots\right)\in X$ since
\[
\left(\frac{1}{k+1}-\frac{1}{k}\right)_{k\in \mathbb{N}}=\left(-\frac{1}{k(k+1)}\right)\in l_1 \quad \text{and} \quad \left(\frac{1}{k}\right)_{k\in \mathbb{N}}\in c_0.
\]
\item $\left(\frac{(-1)^k}{\ln(k+1)}\right)_{k\in \mathbb{N}}\notin X $ since
\begin{align*}
&\left(\frac{\frac{(-1)^{k+1}}{\ln(k+2)}}{(k+1)}-\frac{\frac{(-1)^k}{\ln(k+1)}}{k}\right)_{k\in \mathbb{N}}
=\left(\frac{\frac{(-1)^{k+1}}{\ln(k+2)}}{(k+1)}+\frac{\frac{(-1)^{k+1}}{\ln(k+1)}}{k}\right)_{k\in \mathbb{N}}\notin l_1 \\
& \text{although} \quad \left(\frac{(-1)^k}{k\ln(k+1)}\right)\in c_0.
\end{align*}
\item $(k)_{k\in \mathbb{N}}\notin X$ since
\[
\left(\frac{k+1}{k+1}-\frac{k}{k}\right)_{k\in \mathbb{N}}=\left(0,0,0,\dots\right)\in l_1 \quad \text{but} \quad \left(\frac{k}{k}\right)_{k\in \mathbb{N}}=\left(1,1,1,\dots\right)\notin c_0.
\]
\item $\left(k^2\right)_{k\in \mathbb{N}}\notin X $ since
\[
\left(\frac{(k+1)^2}{k+1}-\frac{k^2}{k}\right)_{k\in \mathbb{N}}=\left(1,1,1,\dots\right)\notin l_1 \quad \text{and} \quad \left(\frac{k^2}{k}\right)_{k\in \mathbb{N}}=(k)_{k\in \mathbb{N}}\notin c_0.
\]
\end{enumerate}
\end{exmps}










The following example relates the well-known class of power sequences to the space $X$.

\begin{exmp}[Power Sequences]\label{exmpse2}\ \\
For $p\in \R$,
\[
\left(k^p\right)_{k\in \N}\in X
\iff p<1.
\]

In particular,
\[
\left(1,1,1,\dots\right),\ \left(\sqrt{k}\right)_{k\in \N}\in X
\quad \text{and}\quad
\left(k\right)_{k\in \N},\ \left(k^2\right)_{k\in \N}\notin X.
\]

Indeed,
\[
\left(\frac{k^p}{k}\right)_{k\in \N}=\left(k^{p-1}\right)_{k\in \N}\in c_0\iff p-1<0
\iff p<1.
\]

Further, since, for $p-1<0$,
    \begin{align*}
    \sum_{k=1}^\infty\left|\frac{(k+1)^p}{k+1}-\frac{k^p}{k}\right|&=\sum_{k=1}^\infty\left|(k+1)^{p-1}-k^{p-1}\right|=\sum_{k=1}^\infty\left(k^{p-1}-(k+1)^{p-1}\right)\\
&=\lim_{n\to \infty}\sum_{k=1}^n\left(k^{p-1}-(k+1)^{p-1}\right)=1-\lim_{n\to \infty}(n+1)^{p-1}=1,
    \end{align*}
we infer that $\left(k^p\right)_{k\in \N}\in X$.

In particular, 
\[
\forall\, p<1: \ \left\|\left(k^p\right)_{k\in \N}\right\|=1.
\]
\end{exmp}

\begin{rem}
As follows from Examples \ref{exmpsse1} and Example \ref{exmpse2} (see also Remark \ref{c0clinf}),
\[
l_\infty\supset c\supset c_0\ni \left(\frac{(-1)^k}{\ln(k+1)}\right)_{k\in \N}\notin X\ni \left(\sqrt{k}\right)_{k\in \N}\notin  l_\infty,
\]

Hence,
\[
 c_0\not\subseteq X\not\subseteq  l_\infty.
\]

Thus, we conclude that the space $X$ is not a part of the hierarchy of the classical 
spaces $c_0$, $c$, and $l_\infty$.
\end{rem}
\subsection{Isometric Isomorphisms}\label{secii}\

Let us look at the following classical space
\[
bv_0:=\left\{(x_k)_{k\in \N}\in c_0\,\middle|\, \sum_{k=1}^\infty\left|x_{k+1}-x_k\right|<\infty\right\},
\]
of \textit{vanishing sequences of bounded variation} with the norm 
\[
bv_0\ni x:=(x_k)_{k\in \N}\mapsto \|x\|_0:= \sum_{k=1}^\infty\left|x_{k+1}-x_k\right|.
\]
(see, e.g., \cite{Dun-SchI}).


\begin{prop}[Isometric Isomorphisms Proposition]\label{IIP}\

\begin{enumerate}[label=\arabic*.]
\item The space $(X,\|\cdot\|) $ is isometrically isomorphic to the space $bv_0$ under the mapping
\[
X\ni x:=(x_k)_{k\in \N}\mapsto
J_1x:=\left(\frac{x_{k}}{k}\right)_{k\in \N}\in bv_0
\]
with the inverse
\[
bv_0\ni y:=(y_k)_{k\in \N}\mapsto J_1^{-1}y:=\left(ky_k\right)_{k\in \N}\in X.
\]
\item The space $bv_0$ is isometrically isomorphic to the space $l_1$ under the mapping
\[
bv_0\ni y:=(y_k)_{k\in \N}\mapsto J_2y:=\left(y_{k+1}-{y_k}\right)_{k\in \N}\in l_1,
\]
with the inverse
\[
l_1\ni z:=(z_k)_{k\in \N}\mapsto J_2^{-1}z:=\left(\sum_{j=0}^{k-1}z_j\right)_{k\in \N}\in bv_0,
\]
where 
\[
z_0:=-\sum_{k=1}^\infty z_k.
\]
\item The space $(X,\|\cdot\|) $ is isometrically isomorphic to the space $l_1$ under the mapping $J:=J_2J_1$, i.e.,
\[
X\ni x:=(x_k)_{k\in \N}\mapsto Jx:=\left(\frac{x_{k+1}}{k+1}-\frac{x_k}{k}\right)_{k\in \N}\in l_1
\]
with the inverse $J^{-1}=J_1^{-1}J_2^{-1}$, i.e.,
\[
l_1\ni z:=(z_k)_{k\in \N}\mapsto J^{-1}z:=\left(k\sum_{j=0}^{k-1}z_j\right)_{k\in \N}\in X,
\]
where
\[
z_0:=-\sum_{k=1}^\infty z_k.
\]
\end{enumerate}
\end{prop}

\begin{proof}\
\begin{enumerate}[label=\arabic*.]
\item As is easily seen, $J_1:X\to bv_0$ is an \textit{injective linear operator} which is a bijection between $X$ and $bv_0$ since
\[
\forall\, y:=\left(y_k\right)_{k\in \N}\in bv_0:\
x:=(ky_k)_{k\in \N} \in X
\ \text{and}\ J_1x=\left(\frac{ky_k}{k}\right)_{k\in \N}=y.
\]

Since further
\begin{align*}
\forall\, x:=\left(x_k\right)_{k\in \N}\in X:\ &
\|J_1x\|_0
=\left\|\left(\frac{x_k}{k}\right)_{k\in \N}\right\|_0\\
&=\sum_{k=1}^{\infty}\left|\frac{x_{k+1}}{k+1}-\frac{x_k}{k}\right|=\|x\|,
\end{align*}
we infer that $J_2$ is an \textit{isometric isomorphism} between the spaces $X$ and $bv_0$ with the inverse
\[
bv_0\ni y:=(y_k)_{k\in \N}\mapsto J_1^{-1}y:=\left(ky_k\right)_{k\in \N}\in X.
\]
\item  As is easily seen, $J_2:bv_0\to l_1$ is a linear operator. 

Suppose that for an arbitrary $y:=(y_k)_{k\in \N}\in bv_0$,
\[
J_2y=\left(y_{k+1}-{y_k}\right)_{k\in \N}=\left(0,0,0\dots\right).
\]

Then, inductively, 
\[
\forall{k}\in \N: \, y_k=y_1,
\]
i.e., $y$ is a constant sequence.

Thus, in view of the fact that
\[
y_k\to 0, \, k\to \infty,
\]
implies that $y=\left(0,0,0\dots\right)$.

Therefore, the linear operator $J_2$ is \textit{injective} (see, e.g., \cite{Markin2018EFA,Markin2020EOT}).

For an arbitrary $z:=\left(z_k\right)_{k\in \N}\in l_1 $, let
\[
z_0:=-\sum_{k=1}^{\infty}z_k\in \F .
\]

Then
\[
y:=\left(\sum_{j=0}^{k-1}z_j\right)_{k\in \N}\in bv_0
\]
since
\[
\lim_{k\to\infty}y_k=\lim_{k\to\infty}\sum_{j=0}^{k-1}z_j=z_0+\sum_{j=1}^{\infty}z_j=0,
\]
and
\[
\sum_{k=1}^{\infty}|y_{k+1}-y_k|=
\sum_{k=1}^{\infty}\left|\sum_{j=0}^{k}z_j-\sum_{j=0}^{k-1}z_j\right|=\sum_{k=1}^{\infty}|z_k|<\infty.
\]

Moreover, 
\[
J_2y=\left(\sum_{j=0}^kz_j-\sum_{j=0}^{k-1}z_j\right)_{k\in \N}=(z_k)_{k\in \N}=z.
\]

Hence, the operator $J_2$ is a \textit{bijection} between $bv_0$ and $l_1$.

Since further
\begin{align*}
\forall\, y:=(y_k)_{k\in \N}\in bv_0:\ &
\|J_2y\|_1=\left\|\left(y_{k+1}-y_k\right)_{k\in \N}\right\|_1\\
&=\sum_{k=1}^{\infty}\left|y_{k+1}-y_k\right|=\|y\|_0,
\end{align*}
where $\|\cdot\|_1 $ is the norm of $l_1$, we infer that $J_2$ is an \textit{isometric isomorphism} between the spaces $bv_0$ and $l_1$ with the inverse
\[
l_1\ni z:=(z_k)_{k\in \N}\mapsto J_2^{-1}z:=\left(\sum_{j=0}^{k-1}z_j\right)_{k\in \N}\in bv_0,
\]
where 
\[
z_0:=-\sum_{k=1}^\infty z_k.
\]
\item By parts $1$ and $2$, it is clear that $J:=J_2J_1$ is an \textit{isometric isomorphism} between $X$ and $l_1$, and hence,
\begin{align*}
\forall\, x:=\left(x_k\right)_{k\in \N}\in X: \ &Jx=J_2J_1x=J_2\left(\frac{x_k}{k}\right)_{k\in \N}\\
&=\left(\frac{x_{k+1}}{k+1}-\frac{x_k}{k}\right)_{k\in\N}\in l_1.    
\end{align*}

Moreover, $J^{-1}:=J_1^{-1}J_2^{-1}$ is the inverse of $J$, and thus,
\begin{align*}
\forall\, z:=\left(z_k\right)_{k\in \N}\in l_1: \ &J^{-1}z=J_1^{-1}J_2^{-1}z=J_1^{-1}\left(\sum_{j=0}^{k-1}z_j\right)_{k\in \N}\\
&=\left(k\sum_{j=0}^{k-1}z_j\right)_{k\in \N} \in X,
\end{align*}
where
\[
z_0:=-\sum_{k=1}^\infty z_k.
\]
\end{enumerate}
\end{proof}

\begin{rems}\label{remsIIP}\
\begin{itemize}
\item The fact that the space $bv_0$ is \textit{isometrically isomorphic} to $l_1$ is stated as  \cite[Exercise IV.$13.11$]{Dun-SchI} (see also \cite{de Malafosse-Malkowsky-Rakocevic}).
\item Since $l_1$ is an infinite dimensional separable Banach space, as follows from the prior proposition, so are the spaces $bv_0$ and $(X,\|\cdot\|)$.
\item Since the subspace $c_{00}$ of eventually zero sequences is dense in the space $l_1$ and $J_2^{-1}(c_{00})=J^{-1}(c_{00})=c_{00}$, by the prior proposition, $c_{00}$ is also dense in the spaces $bv_0$ and $(X,\|\cdot\|)$.
\end{itemize}
\end{rems}

\subsection{Continuous and Dense Embeddings}\

The latter remark is consistent with the following proposition further relating the space $(X,\|\cdot\|)$ to the spaces $l_1$ and $bv_0$.

\begin{prop}[Continuous and Dense Embeddings]\label{CDE}\ \\
The chain of proper inclusions
\begin{equation}\label{DCE4}
c_{00}\subset l_1\subset bv_0\subset X.
\end{equation}
holds, where the embeddings 
\[
l_1\hookrightarrow bv_0\quad \text{and}\quad bv_0\hookrightarrow X
\]
are continuous and dense with
\begin{equation}\label{DCE5}
 \|z\|_0\le 2\|z\|_1, \ z\in l_1, \quad \text{and} \quad \|y\|\le \|y\|_0,\ y\in bv_0.    
\end{equation}

Furthermore, the embedding $ l_1\hookrightarrow X $ is a continuous and dense  embedding with
\begin{equation}\label{DCE6}
\|z\|\le \|z\|_1, \ z\in l_1. 
\end{equation}
\end{prop}
\begin{proof}
First, it is clear that $ c_{00}\subset l_1$ is a proper inclusion. Moreover, 

\[
\forall\, z:=(z_k)_{k\in \N}\in l_1:\ z\in c_0
\]
and
\begin{equation}\label{DCE1}
\sum_{k=1}^\infty\left|z_{k+1}-z_k\right|\le \sum_{k=1}^\infty|z_{k+1}|+\sum_{k=1}^\infty|z_k|\le 2\|z\|_1<\infty.
\end{equation}

Hence, the inclusion $l_1\subseteq bv_0$ holds.

Moreover,
\[
bv_0\ni \left(\frac{1}{k}\right)_{k\in \N}\notin l_1,
\]
since $\left(\frac{1}{k}\right)_{k\in \N}\in c_0$ and
\[
\sum_{k=1}^\infty\frac{1}{k}=\infty\quad \text{while} \quad \sum_{k=1}^\infty\left|\frac{1}{k+1}-\frac{1}{k}\right|
=\sum_{k=1}^\infty\left(\frac{1}{k}-\frac{1}{k+1}\right)=1<\infty.
\]

Therefore, the inclusion $l_1\subset bv_0$ is proper.

Due to estimate \eqref{DCE1},
\[
 \|z\|_0\le 2\|z\|_1, \ z\in l_1,
\]
which makes the embedding $l_1\hookrightarrow bv_0$
\textit{continuous}.

Consider an arbitrary $y:=(y_k)_{k\in \N}\in bv_0$.

In view of $bv_0\subseteq c_0$,
\[
\left(\frac{y_k}{k}\right)_{k\in \N}\in c_0.
\]

Furthermore, since $y\in bv_0$,
\[
\left(r_k:=\sum_{j=k}^\infty|y_{j+1}-y_j|\right)_{k\in \N}\in c_0\left(\R_+\right).
\]

Hence, the following is clear
\begin{equation}\label{1}
\forall\, k\in \N: |y_{k+1}-y_k|=r_k-r_{k+1}
\end{equation}
and
\begin{equation}\label{2}
0\le |y_k|=\sum_{j=k}^\infty\left(|y_j|-|y_{j+1}|\right)\le \sum_{j=k}^\infty|y_{j+1}-y_j|=r_k.
\end{equation}

In view of $r_1=\|y\|_0$, using \eqref{1} and \eqref{2}, for an arbitrary $n\in \N$, we deduce
\begin{align*}
\sum_{k=1}^n\left|\frac{y_{k+1}}{k+1}-\frac{y_k}{k}\right|&=\sum_{k=1}^\infty\left|\frac{y_{k+1}}{k+1}-\frac{y_k}{k+1}+\frac{y_k}{k+1}-\frac{y_k}{k}\right|\\
&=\sum_{k=1}^n\left|\frac{y_{k+1}-y_k}{k+1}-\frac{y_k}{k(k+1)}\right|\\
&\le \sum_{k=1}^n\left(\frac{|y_{k+1}-y_k|}{k+1}+\frac{|y_k|}{k(k+1)}\right)=\sum_{k=1}^n\left(\frac{r_k-r_{k+1}}{k+1}+\frac{|y_k|}{k(k+1)}\right)\\
&\le \sum_{k=1}^n\left(\frac{r_k-r_{k+1}}{k+1}+\frac{r_k}{k(k+1)}\right)\\
&=\sum_{k=1}^n\left(\frac{r_k}{k+1}-\frac{r_{k+1}}{k+1}+\frac{r_k}{k}-\frac{r_k}{k+1}\right)\\
&=\sum_{k=1}^n\left(\frac{r_k}{k}-\frac{r_{k+1}}{k+1}\right)=r_1-\frac{r_{n+1}}{n+1}=\|y\|_0-\frac{r_{n+1}}{n+1}\le \|y\|_0.
\end{align*}

Hence, passing to the limit as $n\to \infty$,

\[
\|y\|=\sum_{k=1}^\infty\left|\frac{y_{k+1}}{k+1}-\frac{y_k}{k}\right|\le \|y\|_0.
\]

Thus, the inclusion $ bv_0\subseteq X$ holds and
\[
\|y\|\le \|y\|_0, \ y\in bv_0,
\]
so that the embedding $bv_0\hookrightarrow X$ is \textit{continuous}.

Moreover, since
\[
X\ni \left(1,1,1,\dots\right)\notin c_0\supseteq bv_0
\]
(see Examples \ref{exmpsse1}), the inclusion $bv_0\subset X $ is proper.

Thus, chain of proper inclusions \eqref{DCE4} holds, the embeddings
\[
l_1\hookrightarrow bv_0\quad \text{and}\quad bv_0\hookrightarrow X
\]
being \textit{continuous} with estimates \eqref{DCE5} in place.

Since  $\left\{e_n:=\left(\delta_{nk}\right)_{k\in \N}\right\}_{n\in \N}$ is a \textit{Schauder basis} for $l_1$ (see Section \ref{seccfcss}), $bv_0$, and $(X,\|\cdot\|)$ (see Section \ref{secconv}) we infer that
\[
c_{00}=\spa\left(\left\{e_n:=\left(\delta_{nk}\right)_{k\in \N}\right\}_{n\in \N}\right)
\]
is a \textit{dense} subspace in $l_1$, $bv_0$, and $(X,\|\cdot\|)$.

Thus, the continuous embeddings
\[
l_1\hookrightarrow bv_0\quad \text{and}\quad bv_0\hookrightarrow X
\]
are also \textit{dense}.

The continuity and denseness of the embeddings 
\[
l_1\hookrightarrow bv_0\quad \text{and}\quad bv_0\hookrightarrow X
\]
instantly implies, the \textit{continuity} and \textit{denseness} for the embedding 
\[
l_1\hookrightarrow X,
\]
in which, by estimates \eqref{DCE5},
\[
\|z\|\le \|z\|_0
\le 2\|z\|_1,\ z\in l_1.
\]

Let us show that the latter estimate can be refined to estimate \eqref{DCE6}. Indeed,
\begin{align*}
\forall\, z:=(z_k)_{k\in \N}\in l_1, \ \forall\, n\in \N,\ n\ge 2:\ &\sum_{k=1}^{n}\left|\frac{z_{k+1}}{k+1}-\frac{z_k}{k}\right|
\le \sum_{k=1}^n\left|\frac{z_{k+1}}{k+1}\right|+\sum_{k=1}^n\left|\frac{z_k}{k}\right|\\
&=|z_1|+\frac{|z_{n+1}|}{n+1}+2\sum_{k=2}^n\frac{|z_k|}{k}\\
&\le |z_1|+|z_{n+1}|+2\sum_{k=2}^n\frac{|z_k|}{2}\\
&=|z_1|+|z_{n+1}|+\sum_{k=2}^n|z_k|=\sum_{k=1}^{n+1}|z_k|.
\end{align*}

Hence, 
\[
\|z\|:=\sum_{k=1}^{\infty}\left|\frac{z_{k+1}}{k+1}-\frac{z_k}{k}\right|
=\lim_{n\to \infty}\sum_{k=1}^{n}\left|\frac{z_{k+1}}{k+1}-\frac{z_k}{k}\right|
\le \lim_{n\to \infty}\sum_{k=1}^{n+1}|z_k|
=\sum_{k=1}^{\infty}|z_k|=\|z\|_1.
\]
\end{proof}

\begin{exmp}[Exponential Sequences]\label{exps}\ \\
For $\lambda\in \C$, 
\[
 \left(\lambda^k\right)_{k\in \mathbb{N}}\in X
 \iff
 |\lambda|<1 \quad \text{or} \quad \lambda=1.
\]

In particular,
\[
\left(1,1,1,\dots\right),\ \left(2^{-k}\right)_{k\in \N}\in X \quad \text{and} \quad 
\left(i^k\right)_{k\in \N},\
\left((-1)^k\right)_{k\in \N},\
\left(2^k\right)_{k\in \N}\notin X.
\]
($i$ is the \textit{imaginary unit}).

Indeed, for $\lambda=1$,
\[
\left(\lambda^k\right)_{k\in \N}=\left(1,1,1,\dots\right)_{k\in \N}\in X 
\]

(see Examples \ref{exmpsse1}). 

For $|\lambda|<1$, since
\[
\left(\lambda^k\right)_{k\in \N}\in l_1,
\]

by the prior proposition, $\left(\lambda^k\right)_{k\in \mathbb{N}}\in X$.

For $|\lambda|>1$, since 
\[
\left(\frac{\lambda^k}{k}\right)\notin c_0,
\]
we infer that $\left(\lambda^k\right)_{k\in \mathbb{N}}\notin X$.

For $\lambda=-1$, 
\[
\sum_{k=1}^\infty\left|\frac{(-1)^{k+1}}{k+1}-\frac{(-1)^k}{k}\right|=\sum_{k=1}^\infty\left|(-1)^{k+1}\right|\left(\frac{1}{k+1}+\frac{1}{k}\right)=\sum_{k=1}^\infty\left(\frac{1}{k+1}+\frac{1}{k}\right)=\infty,
\]

so that $\left(\lambda^k\right)_{k\in \N}=\left((-1)^k\right)_{k\in \N}\notin X$.

Lastly, for $|\lambda|=1$ and $\lambda\notin \{-1,1\}$,

\[
\exists\, \theta\in (-\pi,\pi]\backslash \{0,\pi\}: \ \lambda=e^{i\theta}=\cos \theta +i\sin \theta.
\]

Hence, 
\begin{align*}
\forall\, k\in\N:\ \left|\frac{e^{i\theta (k+1)}}{k+1}-\frac{e^{i\theta k}}{k}\right|&=\left|e^{i\theta k}\right|\left|\frac{e^{i\theta}}{k+1}-\frac{1}{k}\right|=\left|\frac{\cos \theta+i\sin \theta}{k+1}-\frac{1}{k}\right|\\
&=\left|\left(\frac{\cos\theta}{k+1}-\frac{1}{k}\right)+i\frac{\sin\theta}{k+1}\right|\ge \frac{|\sin\theta|}{k+1}.
\end{align*}

Moreover, since $\theta\in (-\pi,\pi]\backslash \{0,\pi\}$ so that $|\sin\theta|\ne 0$, by the \emph{Comparison Test},
\[
\sum_{k=1}^\infty\left|\frac{e^{i\theta (k+1)}}{k+1}-\frac{e^{i\theta k}}{k}\right|=\infty.
\]

Therefore, $\left(\lambda^k\right)_{k\in \N}\notin X$.
\end{exmp}

\subsection{Convergence}\label{secconv}\

The set $\left\{e_n:=\left(\delta_{nk}\right)_{k\in \N}\right\}_{n\in \N}$,
where $\delta_{nk}$ is the \textit{Kronecker delta}, is a \textit{Schauder basis} for the space $l_1$ (see Preliminaries).

Hence, by the \emph{Isometric Isomorphisms Proposition} (Proposition \ref{IIP}),
\[
J^{-1}\left(\left\{e_n\right\}_{n\in \N}\right)
=\left\{J^{-1}e_n\right\}_{n\in \N}
\]
is a Schauder basis for $X$ ($J: X\to l_1$ is the corresponding isometric isomorphism), where
\begin{align*}
\forall\, n\in \N:\  J^{-1}e_n:=&\left(k\left(-\sum_{j=1}^\infty\delta_{nj}+\sum_{j=1}^{k-1}\delta_{nj}\right)\right)_{k\in \N}=\left(k\left(-\sum_{j=k}^{\infty}\delta_{nj}\right)\right)_{k\in \N}\\
=&(-1,-2,\dots,-n,0,0,\dots)= -\sum_{k=1}^nke_k.
\end{align*}

Furthermore, the set 
$\left\{e_n:=\left(\delta_{nk}\right)_{k\in \N}\right\}_{n\in \N}$, is also a \textit{Schauder basis} for $X$ and
\begin{equation}\label{cf}
\forall\, x:=(x_k)_{k\in \N}\in X:\ x=\sum_{k=1}^\infty x_ke_k     
\end{equation}
since
\begin{equation}\label{rmndr}
\begin{aligned}
\forall\, x:=(x_k)_{k\in \N}\in X:\
&\left\|x-\sum_{k=1}^n x_ke_k\right\|=\left\|\sum_{k=n+1}^\infty x_ke_k\right\|\\
&=\frac{|x_{n+1}|}{n+1}+\sum_{k=n+1}^\infty\left|\frac{x_{k+1}}{k+1}-\frac{x_k}{k}\right|\to 0, \ n\to \infty. 
\end{aligned}
\end{equation}

\begin{rem}
Similarly, by the \emph{Isometric Isomorphisms Proposition} (Proposition \ref{IIP}), 
\[
J_2^{-1}\left(\left\{e_n\right\}_{n\in \N}\right)=\left\{-\sum_{k=1}^ne_k\right\}_{n\in \N}
\]
is a \textit{Schauder basis} for the space $bv_0$ ($J_2: bv_0\to l_1$ is the corresponding \textit{isometric isomorphism}) as well as the set 
$\left\{e_n:=\left(\delta_{nk}\right)_{k\in \N}\right\}_{n\in \N}$. 
\end{rem}

Based on the above, we obtain the following

\begin{cor}[Characterization of Convergence in $(X,\|\cdot\|)$]\label{CC}\ \\
A sequence $\left(x^{(n)}:=\left(x_k^{(n)}\right)_{k\in \N}\right)_{n\in \N}$ converges to $x:=(x_k)_{k\in \N}$ in the space
$(X,\|\cdot\|)$, i.e.,
\[
x^{(n)}\to x, \ n\to \infty,
\]
iff
\begin{enumerate}[label=(\arabic*)]
\item $\forall\, k\in \N:\ x_k^{(n)}\to x_k,\ n\to\infty$, and
\item 
$\displaystyle
\forall\,\eps>0\space\ \exists\, K\in \N\ \forall\, n\in \N: \ \sum_{k=K+1}^{\infty}\left|\frac{x_{k+1}^{(n)}}{k+1}-\frac{x_k^{(n)}}{k}\right|<\eps$.
\end{enumerate}
\end{cor}

\begin{proof}
Since the set $\left\{e_n:=\left(\delta_{nk}\right)_{k\in \N}\right\}_{n\in \N} $ is a \textit{Schauder basis} for the space $(X,\|\cdot\|)$ with the coordinate functionals
\[
X\ni x:=(x_k)_{k\in \N}\mapsto c_k(x)=x_k,\ k\in \N,
\]
(see \eqref{cf}), by the \emph{General Characterization of Convergence Proposition} (Theorem \ref{GCC}), a sequence $\left(x^{(n)}:=\left(x_k^{(n)}\right)_{k\in \N}\right)_{n\in \N}$ converges to $x:=(x_k)_{k\in \N}$ in the space
$(X,\|\cdot\|)$, i.e.,
\[
x^{(n)}\to x, \ n\to \infty,
\]
iff
\begin{enumerate}
\item[$(1')$] $\forall\, k\in \N: \ c_k\left(x^{(n)}\right)=x_k^{(n)}\to c_k(x)=x_k$, $n\to \infty$ and
\item[$(2')$] $\forall\, \eps>0\ \exists\, K_0\in \N\ \forall\, K\ge K_0 \ \forall\, n\in \N:$
\[
\left\|\sum_{k=K+1}^\infty x_k^{(n)}e_k\right\|=\frac{\left|x_{K+1}^{(n)}\right|}{K+1}+\sum_{k=K+1}^\infty\left|\frac{x_{k+1}^{(n)}}{k+1}-\frac{x_k^{(n)}}{k}\right|<\eps
\]
(see \eqref{rmndr}).
\end{enumerate}

Thus, condition $(1)$ is equivalent to condition $(1')$.

It is clear that condition $(2')$ implies seemingly weaker condition $(2)$.

For the reverse implication, by condition $(2)$, suppose that
\[
\displaystyle
\forall\,\eps>0\space\ \exists\, K_0\in \N\ \forall\, n\in \N: \ \sum_{k=K_0+1}^{\infty}\left|\frac{x_{k+1}^{(n)}}{k+1}-\frac{x_k^{(n)}}{k}\right|<\frac{\eps}{2}.
\]

Then, since $x^{(n)}\in X$, $n\in\N$, we further have:
\begin{equation*}
\begin{aligned}
\forall\, K\ge K_0\ \forall\, n\in \N:\ \frac{\eps}{2}&>\sum_{k=K_0+1}^{\infty}\left|\frac{x_{k+1}^{(n)}}{k+1}-\frac{x_k^{(n)}}{k}\right|\ge \sum_{k=K+1}^{\infty}\left|\frac{x_{k+1}^{(n)}}{k+1}-\frac{x_k^{(n)}}{k}\right|\\
&\ge \sum_{k=K+1}^\infty \left(\frac{\left|x_k^{(n)}\right|}{k}-\frac{\left|x_{k+1}^{(n)}\right|}{k+1}\right)\\
&=\frac{\left|x_{K+1}^{(n)}\right|}{K+1}-\lim_{m\to \infty}\frac{\left|x_{m+1}^{(n)}\right|}{m+1}=\frac{\left|x_{K+1}^{(n)}\right|}{K+1},
\end{aligned}    
\end{equation*}
and hence, 
\[
\frac{\left|x_{K+1}^{(n)}\right|}{K+1}+\sum_{k=K+1}^{\infty}\left|\frac{x_{k+1}^{(n)}}{k+1}-\frac{x_k^{(n)}}{k}\right|<\frac{\eps}{2}+\frac{\eps}{2}=\eps.
\]

Therefore, condition $(2')$ is implied by condition $(2)$ so that both conditions are equivalent.
\end{proof}

\begin{rem}\label{remCC}
Condition (2) can be equivalently restated as follows:
\[
\sup_{n\in \N}\sum_{k=K+1}^{\infty}\left|\frac{x_{k+1}^{(n)}}{k+1}-\frac{x_k^{(n)}}{k}\right|\to 0,\ K\to \infty.
\]
\end{rem}

\section{Bounded Weighted Backward Shifts}

\begin{lem}[Norm Identities]\label{NI}\ \\
On the space $(X,\|\cdot\|)$, the backward shift
\begin{equation}\label{A}
X \ni x:=\left(x_k\right)_{k\in \N} \mapsto Ax:= \left(x_{k+1}\right)_{k\in \N} \in X
\end{equation}
is a bounded linear operator subject to the following norm identities:
\[
\left\|A^n\right\|= n+1,\ n\in \N.
\]

In particular, $\|A\|=2$.
\end{lem}

\begin{proof}
Let
\begin{equation}\label{bsbv0}
bv_0\ni y:=(y_k)_{k\in \N}\mapsto \hat{A}y:=(y_{k+1})_{k\in \N}\in bv_0.
\end{equation}

The fact that the mapping $\hat{A}$ is well defined and \textit{linear} on $bv_0$ is obvious.

The linear operator $\hat{A}$ is also \textit{bounded} since
\[
\forall\, y:=(y_k)_{k\in \N}:\ \sum_{k=1}^\infty|y_{k+2}-y_{k+1}|=\sum_{k=2}^\infty|y_{k+1}-y_k|
\le \sum_{k=1}^\infty|y_{k+1}-y_k|=\|y\|_0.
\]

In particular, we infer
\begin{equation}\label{0norm}
\left\|\hat{A}\right\|\le 1.
\end{equation}

Let $x:=(x_k)_{k\in \N}\in X$ be arbitrary. 

By the \emph{Isometric Isomorphisms Proposition} (Proposition \ref{IIP}),
\begin{equation}\label{Xbv0}
\exists\, y:=(y_k)_{k\in \N}\in bv_0:\ x=J_1^{-1}y:=(ky_k)_{k\in \N}
\end{equation}
($J_1:X\to bv_0$ is the corresponding \emph{isometric isomorphism}).

Hence, for any $n\in \N$,
\begin{multline}\label{Xrep}
(x_{k+n})_{k\in \N}
\hfill\text{by \eqref{Xbv0};}
\\
\shoveleft{
=\left((k+n)y_{k+n}\right)_{k\in \N}=(ky_{k+n})_{k\in \N}+n(y_{k+n})_{k\in \N}
}
\\
\hfill\text{by the \emph{Isometric Isomorphisms Proposition} (Propostion \ref{IIP})
and \eqref{bsbv0};}
\\
\shoveleft{
=J_1^{-1} \hat{A}^ny+n \hat{A}^ny
\hfill\text{since $n \hat{A}^ny\in bv_0\subset X$}
}
\\
\hfill\text{by the \emph{Continuous and Dense Embeddings Proposition} (Proposition \ref{CDE})}
\\
\hfill\text{and \eqref{Xbv0};}
\\
\shoveleft{
=J_1^{-1} \hat{A}^nJ_1x+n \hat{A}^nJ_1x\in X.
}
\hfill
\end{multline}

Hence, the operator $A$ is well defined on $X$ by \eqref{A} and is, obviously, \textit{linear}.

Let $x\in X$ with $\|x\|=1$ be arbitrary. 

For any $n\in \N$, by \eqref{Xrep},
\begin{multline*}
\left\|A^nx\right\|=\left\|J_1^{-1} \hat{A}^nJ_1x+n \hat{A}^nJ_1x\right\|\le \left\|J_1^{-1} \hat{A}^nJ_1x\right\|+n\left\| \hat{A}^nJ_1x\right\|
\\
\hfill\text{by the \emph{Continuous and Dense Embeddings Proposition} (Proposition \ref{CDE}),}
\\
\hfill\text{estimates \eqref{DCE5};}
\\
\shoveleft{
\le\left\|J_1^{-1} \hat{A}^nJ_1x\right\|+n\left\| \hat{A}^nJ_1x\right\|_0
}
\\
\hfill\text{by the \emph{Isometric Isomorphisms Proposition} (Proposition \ref{IIP})}
\\
\shoveleft{
=\left\| \hat{A}^nJ_1x\right\|_0+n\left\| \hat{A}^nJ_1x\right\|_0
=\left(n+1\right)\left\| \hat{A}^nJ_1x\right\|_0
\le \left(n+1\right){\left\| \hat{A}\right\|}^n\left\|J_1x\right\|_0
}
\\
\hfill\text{by estimate \eqref{0norm};}
\\
\shoveleft{
\le \left(n+1\right)\left\|J_1x\right\|_0
}
\\
\hfill\text{by the \emph{Isometric Isomorphisms Proposition} (Proposition \ref{IIP})}
\\
\shoveleft{
=\left(n+1\right)\|x\|=n+1.
}
\hfill
\end{multline*}

Therefore, $A$ is a \textit{bounded} linear operator on $(X,\|\cdot\|)$ and
\[
\left\|A^n\right\|\le n+1,\ n\in\N.
\]

Let $n\in \N$ be arbitrary and
\[
x_k^{(n)}:=\begin{cases}
k,& k\le n+1,\\
0,& k>n+1,
\end{cases}\ k\in \N.
\]

Then, in view of the \emph{Continuous and Dense Embeddings Proposition} (Proposition \ref{CDE}),
\[
x^{(n)}:=\left(x_k^{(n)}\right)_{k\in \N}\in c_{00}\subset X.
\]

Since
\[
\left\|x^{(n)}\right\|=\sum_{k=1}^{n}\left|\frac{k+1}{k+1}-\frac{k}{k}\right|+\left|\frac{0}{n+2}-\frac{n+1}{n+1}\right|=1
\]
and 
\[
\left\|A^nx^{(n)}\right\|=\left\|\left(x_{k+n}^{(n)}\right)_{k\in \N}\right\|=\left|\frac{0}{2}-\frac{n+1}{1}\right|=n+1,
\]
we infer that
\[
n+1\le \left\|A^n\right\|,\ n\in\N.
\]

Thus, we conclude that
\[
 \left\|A^n\right\|=n+1,\ n\in\N.
\]

\end{proof}
\begin{rems}\ 
\begin{itemize}
\item By the prior lemma, we conclude that the space $(X,\|\cdot\|) $ is \textit{shift invariant}. 
\item As follows from the proof of the prior lemma, for the backward shift
\begin{equation*}
X \ni x:=\left(x_k\right)_{k\in \N} \mapsto Ax:= \left(x_{k+1}\right)_{k\in \N} \in X
\end{equation*}
on the space $(X,\|\cdot\|)$, the following representations
\[
\forall\, n\in \N:\ A^n=J_1^{-1} \hat{A}^nJ_1+n \hat{A}^nJ_1
\]
hold.
\end{itemize}
\end{rems}

\begin{thm}[Bounded Weighted Backward Shifts]\label{BWBS}\ \\
Let $w\in \F$, the bounded linear weighted backward shift operator 
\[
X\ni x:=\left(x_k\right)_{k\in \N}\mapsto A_wx:=w\left(x_{k+1}\right)_{k\in \N}\in X
\]
on the space $(X,\|\cdot\|)$ is 
\begin{enumerate}[label=(\arabic*)]
\item nonhypercyclic for $|w|<1$, 
\item hypercyclic but not chaotic for $|w|=1$, and
\item chaotic along with every power $A_w^n$, $n\in \N$, for $|w|>1$.
\end{enumerate}

Provided the space $(X,\|\cdot\|)$ is complex (i.e., $\F=\C$),
\[
\sigma\left(A_w\right)=\left\{ \lambda\in \C \,\middle|\, |\lambda|\le |w| \right\}
\]
with
\[
\sigma_p(A_w)=\left\{ \lambda\in \C \,\middle|\, |\lambda|< |w| \right\}\cup \left\{w\right\}
\quad
\text{and}
\quad
\sigma_c(A_w)=\left\{ \lambda\in \C \,\middle|\, |\lambda|=|w| \right\}\setminus \left\{w\right\}.
\]
\end{thm}

\begin{proof}
The fact that $A_w=wA$, where 
\[
X \ni x:=\left(x_k\right)_{k\in \N} \mapsto Ax:= \left(x_{k+1}\right)_{k\in \N} \in X
\]
is the unweighted backward shift, is a \textit{bounded linear operator} on $X$ for arbitrary $w\in \F$ follows from the \emph{Norm Identities Lemma} (Lemma \ref{NI}) by which
\begin{equation}\label{norm}
\forall\, n\in \N: \ \left\|A_w^n\right\|=\left\|w^nA^n\right\|=|w|^n\left\|A^n\right\|= |w|^n\left(n+1\right).
\end{equation}

With this in mind, for the following below consider $w\in \F$.
\begin{enumerate}[label=(\arabic*)]
\item Let $|w|<1 $.

By identities \eqref{norm}, we infer that
\[
\forall\, x\in X:\ 0\le \lim_{n\to \infty} \left\|A_w^nx\right\|\le \lim_{n\to \infty}|w|^n\left(n+1\right)\|x\|=0.
\]

Therefore, $A_w$ is \textit{not} hypercyclic. 
\item Let $|w|=1$.

In view of $A_w$ being a bounded linear operator on $X$, for every $n\in \N$, $A_w^n$ is a bounded linear operator on $X$, and hence, by \emph{Characterization of Closedness for Bounded Linear Operators} (see {\cite{Markin2020EOT,Markin2018EFA}}), is also a \textit{closed linear operator}.

Since $w\ne 0$, let
\[
c_{00}\ni x:=\left(x_k\right)_{k\in \N}\mapsto B_wx:=w^{-1}\left(x_{k-1}\right)_{k\in \N}, 
\]
with $x_0:=0 $ and the linearity being obvious.

Inductively, 
\[
\forall\, x:=(x_k)_{k\in \N}\in c_{00}\ \forall\, n\in \N: \ B_w^nx=w^{-n}\left(x_{k-n}\right)_{k\in \N}\in c_{00}\subset X
\]
($x_{k-n}:=0, \ k=1,\dots, n$).

Furthermore, for arbitrary $x:=(x_k)_{k\in \N}\in c_{00}\subset l_1\subset X$, we have:
\[
\forall\, n\in \N:\ \left(\frac{x_{k-n}}{k}\right)_{k\in \N}\in c_{00}\subset l_1\subset X
\] 
and
\begin{multline*}
\left\|B_w^nx\right\|=\left\|w^{-n}\left(x_{k-n}\right)_{k\in \N}\right\|=|w|^{-n}\left\|\left(x_{k-n}\right)_{k\in \N}\right\|
\\
\hfill\text{since $|w|=1$;}
\\
\shoveleft{
=\left\|\left(x_{k-n}\right)_{k\in \N}\right\|
}\\
\hfill\text{by the \emph{Isometric Isomorphisms Proposition} (Proposition \ref{IIP}) ;}
\\
\shoveleft{
=\left\|\left(\frac{x_{k-n}}{k}\right)_{k\in \N}\right\|_0
}\\
\hfill\text{by the \textit{Continuous and Dense Embeddings Proposition} (Proposition \ref{CDE}),}
\\
\hfill\text{estimates \eqref{DCE5};}
\\
\shoveleft{
\le 2\left\|\left(\frac{x_{k-n}}{k}\right)_{k\in \N}\right\|_1=2\left\|\left(\frac{x_{k}}{k+n}\right)_{k\in \N}\right\|_1\le \frac{2}{n+1}\|x\|_1\to 0,\ n\to \infty.
}
\hfill
\end{multline*}

Considering $x\in c_{00}$, it is clear that 
\[
A_w^nx\to 0,\ n\to \infty.
\]

Therefore,
\begin{equation*}\label{HBS2}
\forall\, x\in c_{00}: \ A_w^nx, B_w^nx\to 0 , \ n\to \infty.
\end{equation*}
i.e., the second condition of the \emph{Sufficient Condition for Hypercyclicity} (Theorem \ref{SCH}) is met. 

On the other hand, the first condition is clearly met. i.e.,
\begin{equation*}\label{HBS3}
\forall\, x:=(x_k)_{k\in \N}\in c_{00}:\ A_wB_wx=w\left((B_wx)_{k+1}\right)_{k\in \N}=(x_k)_{k\in \N}:=x.
\end{equation*}

Furthermore, the following is clear:
\begin{equation*}\label{HBS4}
c_{00}\subseteq C^\infty(A):=\bigcap_{n=1}^\infty D\left(A^n\right)\subseteq X.
\end{equation*}
Therefore, since $c_{00}$ is dense in the space $(X,\|\cdot\|)$ (see Remarks \ref{remsIIP} and Proposition \ref{CDE}), by the \emph{Sufficient Condition for Hypercyclicity} (Theorem \ref{SCH}), $A_w$ is \textit{hypercyclic}.

In \cite{Grosse-Erdmann2000}, it is shown that the only periodic points  of $A$ are \textit{constant sequences}. 

Here, we generalize the proof and obtain the following:
\begin{equation}\label{inclconst}
\Per(A_w)\subseteq \spa\{(1,1,1,\dots)\}.
\end{equation}

Indeed, by way of contradiction, suppose 
\[
\exists\, x:=(x_k)_{k\in \N}\in \Per(A_w)\backslash \spa\{(1,1,1,\dots)\}.
\]

Then, 
\[
\exists\, N\in \N:\ x=A_w^Nx=w^N(x_{k+N})_{k\in \N}
\]
and
\[
\exists\, j\in \N:\ x_j\ne x_{j+1}.
\]

Therefore, 
\[
\forall\, k\in \N:\ w^{kN}x_{j+kN}=x_j\ne x_{j+1}=w^{kN}x_{j+1+kN}
\]
and hence, in view of $|w|=1$, 
\begin{align*}
\infty>\|x\|&\ge \sum_{k=1}^\infty\left|\frac{x_{j+1+kN}}{j+1+kN}-\frac{x_{j+kN}}{j+kN}\right|=\sum_{k=1}^\infty\left|\frac{x_{j+1}w^{-kN}}{j+1+kN}-\frac{x_jw^{-kN}}{j+kN}\right|\\
&=\sum_{k=1}^\infty{\left|w\right|}^{-kN}\left|\frac{x_{j+1}}{j+1+kN}-\frac{x_j}{j+kN}\right|\\
&=\sum_{k=1}^\infty\left|\frac{x_{j+1}}{j+1+kN}-\frac{x_j}{j+kN}\right|\\
&=\sum_{k=1}^\infty\left|\frac{x_{j+1}}{j+1+kN}-\frac{x_j}{j+1+kN}+\frac{x_j}{j+1+kN}-\frac{x_j}{j+kN}\right|\\
&=\sum_{k=1}^\infty\left|\frac{x_{j+1}-x_j}{j+1+kN}-\frac{x_j}{(j+kN)(j+1+kN)}\right|.
\end{align*}

However, since $x_j\ne x_{j+1}$,
\[
\lim_{k\to \infty}\frac{\left|\frac{x_{j+1}-x_j}{j+1+kN}\right|}{\left|\frac{x_{j+1}-x_j}{j+1+kN}-\frac{x_j}{(j+kN)(j+1+kN)}\right|}=1
\]
so that, by the \emph{Comparison Test},
\[
\sum_{k=1}^\infty\left|\frac{x_{j+1}-x_j}{j+1+kN}-\frac{x_j}{(j+kN)(j+1+kN)}\right|=\infty.
\]

The obtained contradiction implies inclusion \eqref{inclconst}, which, in view of the fact that
the one-dimensional subspace $\spa\left\{(1,1,1,\dots)\right\}$ of \textit{constant sequences} is \textit{nowhere dense} in the infinite-dimensional space $(X,\|\cdot\|)$ (see, e.g., \cite{Markin2018EFA,Markin2020EOT}), implies that
\[
\bar{\Per(A_w)}\neq X.
\]

Whence, we conclude that the operator $A_w$ is \textit{not} chaotic. 
\item Let $|w|>1$. 

Similar to the case of $|w|=1$, it is clear that $A_w^n$ is a \textit{closed linear operator} for any $n\in \N$.

Further, since $w\ne 0$, let
\[
c_{00}\ni x:=(x_k)_{k\in \N}\mapsto B_wx:= w^{-1}\left(x_{k-1}\right)_{k\in \N}\in c_{00}\subset X
\]

with $x_0:=0$. 

Inductively, 
\[
\forall\, x:=(x_k)_{k\in \N}\in c_{00}\ \forall\, n\in \N: \ B_w^nx=w^{-n}\left(x_{k-n}\right)_{k\in \N}\in c_{00}\subset X
\]
($x_{k-n}:=0$, $k=1,\dots,n$).

Hence, for arbitrary $x:=(x_k)_{k\in \N}\in c_{00}\subset l_1\subset X$ and $n\in \N$,
\begin{multline*}
\left\|B_w^nx\right\|=\left\|w^{-n}\left(x_{k-n}\right)_{k\in \N}\right\|=|w|^{-n}\left\|\left(x_{k-n}\right)_{k\in \N}\right\|
\\
\hfill\text{by the \textit{Continuous and Dense Embeddings Proposition} (Proposition \ref{CDE}),}
\\
\hfill\text{estimate \eqref{DCE6};}
\\
\ \ \
\le|w|^{-n}\left\|\left(x_{k-n}\right)_{k\in \N}\right\|_1=|w|^{-n}\|x\|_1.
\hfill
\end{multline*}

In view of $\|x\|_1<\infty$ being fixed, $\left(A_w^nx\right)_{n\in \N}$ eventually zero, and $|w|>1$, 
\begin{align*}
\forall\, x\in c_{00}\ &\exists \alpha(x)\in \left(|w|^{-1},1\right)\subset (0,1), \ \exists\, c=c(x,\alpha)>\|x\|_1\ge 0 \ \forall\, n\in \N:\\
&\max\left(\left\|A_w^nx\right\|, \left\|B_w^nx\right\|\right)\le \max\left(\left\|A_w^nx\right\|, |w|^{-n}\|x\|_1\right)\le c\alpha^n.
\end{align*}
i.e., the second condition of the \textit{Sufficient Condition for Linear Chaos} (Theorem \ref{SCC}) is met.

Furthermore, 
\begin{equation*}\label{WBS3}
\forall\, x:=(x_k)_{k\in \N}\in c_{00}:\ A_wB_wx=w\left((B_wx)_{k+1}\right)_{k\in \N}=(x_k)_{k\in \N}:=x.
\end{equation*}

i.e., the first condition of the \emph{Sufficient Condition for Linear Chaos} (Theorem \ref{SCC}) is also met.

Lastly, the space $c_{00}\subset X$ is a dense subspace satisfying 
\begin{equation*}\label{WBS4}
c_{00}\subseteq C^{\infty}\left(A_w\right):=\bigcap_{n=1}^\infty D\left(A_w^n\right)\subseteq X.    
\end{equation*}

Hence, by the \emph{Sufficient Condition for Linear Chaos} (Theorem \ref{SCC}) and the \textit{Chaoticity of Powers Corollary} (Corollary \ref{CP}), $A_w$ is a \textit{chaotic} linear operator as well as its every power $A_w^n$ ($n\in \N$).
\end{enumerate}

Let us now analyze the spectral structure of $A$.

For this purpose, suppose that the space $(X,\|\cdot\|)$ is complex.

Then, for $\lambda\in \C\backslash \{0\}$, by the \emph{Exponential Sequences Example} (Example \ref{exps}), 
\begin{align*}
A-\lambda I \ \text{is not injective}& \iff \exists\, x:=(x_k)_{k\in \N}\in X\backslash \{(0,0,0,\dots)\}: \ (A-\lambda I)x=0\\
&\iff \exists\, x:=(x_k)_{k\in \N}\in X\backslash \{(0,0,0,\dots)\}: \ Ax=\lambda x\\
&\iff \exists\, x:=(x_k)_{k\in \N}\in X\backslash \{(0,0,0,\dots)\}: \ x_{k+1}=\lambda x_k\\
& \iff \exists\, x_1\in \C: x_1\left(\lambda^{k-1}\right)_{k\in \N}\in X\backslash \{(0,0,0,\dots)\}\\
& \iff \left(\lambda^{k-1}\right)_{k\in \N}\in X\backslash \{(0,0,0,\dots)\}\\
&\iff \left(\lambda^{k}\right)_{k\in \N}=\lambda \left(\lambda^{k-1}\right)_{k\in \N}\in X\backslash \{(0,0,0,\dots)\}\\
&\iff 0<|\lambda|<1 \quad \text{or}\quad \lambda=1.
\end{align*}

Furthermore, for the case of $\lambda=0$, it is clear that $A$ is \textit{not} injective since for $e _1:=\left(\delta_{1k}\right)_{k\in \N}\in X\backslash \{(0,0,0,\dots)\}$,
\[
Ae_1=(0,0,0,\dots).
\]

Hence,
\begin{equation}\label{pspec}
\sigma_p(A)=\left\{ \lambda\in \C \,\middle|\, |\lambda|< 1 \right\}\cup \left\{1\right\}.
\end{equation}

In view of $A$ being a closed linear operator, $\sigma(A)\subseteq \C$ is a \textit{closed} subset (see, e.g., \cite{Markin2020EOT}) so that
\begin{equation}\label{specinc}
\left\{ \lambda\in \C \,\middle|\, |\lambda|\le 1 \right\}=\overline{\sigma_p(A)} \subseteq \sigma(A).
\end{equation}

Furthermore, since $A$ is a bounded linear operator on the complex Banach space $(X,\|\cdot\|)$, by \emph{Gelfand's Spectral Radius Theorem} (see, e.g., \cite{Markin2020EOT}), considering the \emph{Norm Identities Lemma} (Lemma \ref{NI}),
\begin{equation}\label{gelf}
\begin{aligned}
\max_{\lambda\in \sigma(A)}|\lambda|&=\lim_{n\to \infty}\left\|A^n\right\|^\frac{1}{n}= \lim_{n\to \infty}{\left(n+1\right)}^{\frac{1}{n}}=\lim_{n\to \infty}e^{\frac{\ln\left(n+1\right)}{n}}=e^0=1.
\end{aligned}
\end{equation}

Therefore, in view of \eqref{specinc} and \eqref{gelf}, the following holds:
\begin{equation}\label{spec}
\sigma(A)=\left\{ \lambda\in \C \,\middle|\, |\lambda|\le 1\right\}.
\end{equation}

Since the operator $A$ is \textit{hypercyclic}, by {\cite[Proposition $4.1$]{arXiv:2106.14872}} 
\[
\sigma_r(A)=\emptyset.
\]

Therefore, since $ \sigma_p(A), \, \sigma_c(A),$ and $\sigma_r(A)$ partition $\sigma(A)$,
\begin{equation}\label{cspec}
\sigma_c(A)=\left\{ \lambda\in \C \,\middle|\, |\lambda|=1 \right\}\setminus \left\{1\right\}.
\end{equation}

Hence, provided the space $(X,\|\cdot\|)$ is complex, since $\sigma_p(A_0)=\sigma(A_0)=\{0\}$ is clear, by \eqref{pspec}, \eqref{spec}, and \eqref{cspec},

\begin{equation}\label{spec2}
\forall\, w\in \C:\ \sigma(A_w)=\left\{ \lambda\in \C \,\middle|\, |\lambda|\le |w|\right\}
\end{equation}
with
\begin{equation}\label{pspec2}
\sigma_p(A_w)=\left\{ \lambda\in \C \,\middle|\, |\lambda|< |w| \right\}\cup \left\{w\right\} 
\end{equation}
and
\[
\sigma_c(A_w)= \left\{ \lambda\in \C \,\middle|\, |\lambda|=|w| \right\}\setminus \left\{w\right\}.
\]
\end{proof}

\begin{rems}\ 
\begin{itemize}
\item For $w\in \F$ and $|w|<1$, when the space $(X,\|\cdot\|)$ is complex, by \eqref{spec2},
\[
\sigma(A_w)=\sigma(wA)=w\sigma(A)=\left\{\lambda\in \C \, \middle|\,|\lambda|\le |w|<1 \right\}.
\]

Hence, 
\[
\sigma(A_w)\cap \left\{\lambda\in \C\, \middle|\, |\lambda|= 1\right\}=\emptyset,
\]

so that, by Theorem \ref{KT} (see {\cite[Theorem $5.6$]{Grosse-Erdmann-Manguillot}}), $A_w$ is \textit{not} hypercyclic.
\item For $w\in \F$ and $|w|=1$, the fact that $A_w$ is \textit{not} chaotic coincides with $B_w$ not satisfying the second condition of the \emph{Sufficient Condition for Linear Chaos} (Theorem \ref{SCC}).

Indeed, for $e_1:=(1,0,0,0,\dots)\in c_{00}$, 
\begin{align*}
\left\|B_w^ne_1\right\|&=|w|^n\left\|B_1^ne_1\right\|=\left\|B_1^ne_1\right\|\\
&=\left|\frac{1}{n+1}-\frac{0}{n}\right|+\left|\frac{0}{n+2}-\frac{1}{n+1}\right|=\frac{2}{n+1}.
\end{align*}

However, 
\[
\forall\, \alpha\in (0,1)\ \forall\, c>0:\ \lim_{n\to \infty}\frac{\left\|B_w^ne_1\right\|}{c\alpha^n}=\frac{1}{c}\lim_{n\to \infty}\frac{2}{(n+1)\alpha^n}=\infty.
\]
\item For $w\in \F$ and $|w|=1$, 
\[
\Per(A_w)\subseteq \spa\{(1,1,1,\dots)\}.
\]

From this, the following holds:
\[
\Per(A_w)=\spa\{(1,1,1,\dots)\} \iff \exists\, N\in \N:\ w^N= 1
\]
and
\[
\Per(A_w)=\emptyset \iff \ \forall\, N\in \N:\ w^N\ne 1. 
\]
\item  For $w\in \F$ with $|w|=1$, when the space $(X,\|\cdot\|)$ is complex, by \eqref{pspec2}, 
\[
\sigma_p(A_w)\cap \left\{\lambda\in \C\, \middle|\, |\lambda|=1\right\}=\{w\}.
\]

Hence, $\sigma_p(A_w)$ doesn't contain infinitely many roots of unity so that, by Proposition \ref{P5.7} (see {\cite[Proposition $5.7$]{Grosse-Erdmann-Manguillot}}), $A_w$ is \textit{not} chaotic.

\end{itemize}
\end{rems}

\section{Unbounded Weighted Backward Shifts}

\begin{lem}\label{lem}\ \\
Let $w\in \F$ and $|w|>1$. Then, for the weighted backward shift operator
\[
A_wx:=\left(w^kx_{k+1}\right)_{k\in \N}
\]
in the space $(X,\|\cdot\|)$ with maximal domain
\[
D(A_w):=\left\{ x:=\left(x_k\right)_{k\in \N} \in X \,\middle|\, \left(w^kx_{k+1}\right)_{k\in \N}\in X \right\},
\]
each power 
\[
A_w^{n}x= \left( \left[ \prod_{j=k}^{k+n-1} w^j \right]x_{k+n} \right)_{k \in \N},\ n\in \N,
\]
with domain
\[
D(A_w^{n})= \left\{ x := \left(x_k \right)_{k \in \N} \in X \,\middle|\, \left( \left[ \prod_{j=k}^{k+n-1} w^j \right]x_{k+n} \right)_{k \in \N} \in X \right\}
\]
is a densely defined unbounded closed linear operator and the subspace
\[
C^\infty(A_w):=\bigcap_{n=1}^\infty D(A_w^n)
\]
is dense in $(X,\|\cdot\|)$.
\end{lem}

\begin{proof}
Let $w\in \F$ with $|w|>1$ be arbitrary.

Similarly to \cite[Lemma $3.1$]{arXiv:2203.02032},
\[
\forall\, n\in \N:\ A_w^nx =  \left( \left[ \prod_{j=k}^{k+n-1} w^j \right]x_{k+n} \right)_{k \in \N},  
\]
with domain
\[
D(A_w^{n})= \left\{ x := \left(x_k \right)_{k \in \N} \in X \,\middle|\, \left( \left[ \prod_{j=k}^{k+n-1} w^j \right]x_{k+n} \right)_{k \in \N} \in X \right\},
\]
where
\[
c_{00}\subseteq D(A_w^{n+1}) \subseteq D(A_w^n).
\]

Since $c_{00}\subset X$ is a dense subspace (see Remarks \ref{remsIIP} and Proposition \ref{CDE}), it also follows verbatim {\cite[Lemma $3.1$]{arXiv:2203.02032}}, that each power $A_w^n$ ($n\in \N$) is densely defined with
\[
C^\infty(A_w) := \bigcap_{n=1}^\infty D(A_w^n)
\] 
being \textit{dense} in the space $(X,\|\cdot\|)$.

Let $e_m:=\left(\delta_{mk} \right)_{k\in \N}\in c_{00}\subset X$, $m\in\N$. 

If $m\ge 2$, then
\begin{equation}\label{scaledsu}
\left\|\frac{m}{2}e_m\right\|=\frac{m}{2}\|e_m\|=\frac{m}{2}\left(\left|\frac{1}{m}-\frac{0}{m-1}\right|+\left|\frac{0}{m+1}-\frac{1}{m}\right|\right)=\frac{m}{2}\frac{2}{m}=1.
\end{equation}

For an arbitrary $n\in \N$, in view of $|w|>1$,
\begin{align*}
\forall\, m\ge 2:\ &\left\|A_w^n\frac{m+n}{2}e_{m+n}\right\| = \frac{m+n}{2}\left\|A_w^ne_{m+n}\right\|\\
&=\frac{m+n}{2}\left\| \left( \left[ \prod_{j=k}^{k+n-1} w^j \right] \delta_{(m+n)(k+n)} \right)_{k\in\N}  \right\|\\
& = \frac{m+n}{2}\left(\left|\frac{\displaystyle\prod_{j=m}^{m+n-1} w^j}{m}-\frac{0}{m-1}\right|+\left|\frac{0}{m+1}-\frac{\displaystyle\prod_{j=m}^{m+n-1} w^j}{m}\right|\right)\\
&= \frac{m+n}{m}\prod_{j=m}^{m+n-1} |w|^j\ge \prod_{j=m}^{m+n-1} |w|^j\ge |w|^{mn}\to \infty, \ m\to \infty.
\end{align*}

Therefore, in view of \eqref{scaledsu}, $A_w^n$  is \textit{unbounded}.
		
Let $n\in \N$ and $\left(x^{(m)}:=\left(x_k^{(m)} \right)_{k\in \N} \right)_{m\in \N}$ be a sequence in $D(A_w^n)$ such that
\[
x^{(m)} \to x:=\left(x_k \right)_{m\in \N}\in X,\ m \to \infty,
\]
and
\[
A_w^nx^{(m)} = \left[ \prod_{j=k}^{k+n-1} w^j \right]x_{k+n}^{(m)} \to y:=\left(y_k \right)_{k\in \N}\in X,\ m \to \infty.
\]

Then, by the \emph{Characterization of Convergence in $(X,\|\cdot\|) $ Corollary} (Corollary \ref{CC}), for any $k\in \N$,
\[
x_k^{(m)} \to x_k,\ m \to \infty,
\]
and 
\[
\left[ \prod_{j=k}^{k+n-1} w^j \right]x_{k+n}^{(m)} \to y_k, \ m \to \infty.
\]

Proceeding as in {\cite[Lemma $3.1$]{arXiv:2203.02032}}, we infer that, for each $k\in \N$,
\[
\left[ \prod_{j=k}^{k+n-1} w^j \right]x_{k+n} = y_k,
\]
which means that
\[
\left( \left[ \prod_{j=k}^{k+n-1} w^j \right]x_{k+n} \right)_{k \in \N}=y\in X.
\]
		
Whence, we conclude that $x \in D(A_w^n)$ and $y = A_w^nx$, which, by the \textit{Sequential Characterization of Closed Linear Operators} (see, e.g., \cite{Markin2018EFA,Markin2020EOT}), implies the operator $A_w^n$ is \textit{closed}.
\end{proof}

\begin{thm}[Unbounded Weighted Backward Shifts]\label{UWBS}\ \\
For an arbitrary $w\in \F$ with $|w|>1$, the unbounded linear weighted backward shift operator
\[
A_wx:=\left(w^kx_{k+1}\right)_{k\in \N}
\]
in the space $(X,\|\cdot\|)$ with maximal domain
\[
D(A_w):=\left\{ x:=\left(x_k\right)_{k\in \N} \in X \,\middle|\, \left(w^kx_{k+1}\right)_{k\in \N}\in X \right\}
\]
is chaotic as well as its every power $A_w^n$, $n\in \N$.

Furthermore, each $\lambda \in \F$ is an eigenvalue for $A_w$ of geometric multiplicity $1$, i.e.,
\[
\dim\ker(A_w-\lambda I)=1.
\]
In particular, provided the space $(X,\|\cdot\|)$ is complex,
\[
\sigma_p\left(A_w\right)=\C.
\]
\end{thm}

\begin{proof}
Let $w\in \F$ with $|w|>1$ be arbitrary.

Consider the following mapping:
\[
c_{00}\ni x:=(x_k)_{k\in \N}\mapsto B_wx:=\left(w^{-(k-1)}x_{k-1}\right)_{k\in \N}\in c_{00}\quad (x_0:=0).
\]
For which, it is clear that $B_w$ is a well-defined linear operator such that
\begin{equation}\label{RI2}
A_wB_wx=x,\ x\in Y.
\end{equation}

i.e., the first condition of the \emph{Sufficient Condition for Linear Chaos} (Theorem \ref{SCC}) is met.

Notice
\[
X\ni x:=(x_k)_{k\in \N}\mapsto B_w^2x= \left( w^{-(k-1)} w^{-(k-2)} x_{k-2} \right)_{k \in \N}\quad (x_{k-2}:=0,\ k=1,2)
\]
and 
\[
X\ni x:=(x_k)_{k\in \N}\mapsto B^3x= \left( w^{-(k-1)} w^{-(k-2)} w^{-(k-3)} x_{k-3} \right)_{k \in \N}
\]
($x_{k-3}:=0,\ k=1,2,3$).

Inductively, for any $n\in \N$,
\[
X\ni x:=(x_k)_{k\in \N}\mapsto B^nx = \left(\left[ \prod_{j=1}^{n} w^{-(k-j)} \right] x_{k-n} \right)_{k \in \N}
\]
($x_{k-n}:=0,\ k=1,\dots,n$).

Hence, for arbitrary $x:=(x_k)_{k\in \N}\in c_{00}\subset l_1\subset X$ and $n\in \N$,
\begin{multline*}
\left\|B_w^nx\right\|=\left\| \left(\left[ \prod_{j=1}^{n} w^{-(k-j)} \right] x_{k-n} \right)_{k \in \N}\right\|
\\
\hfill\text{by the \textit{Continuous and Dense Embeddings Proposition} (Proposition \ref{CDE}),}
\\
\hfill\text{estimate \eqref{DCE6};}
\\
\shoveleft{
\le\left\| \left(\left[ \prod_{j=1}^{n} w^{-(k-j)} \right] x_{k-n} \right)_{k \in \N}\right\|_1=\left\| \left(\left[ \prod_{j=1}^{n} w^{-(k+n-j)} \right] x_{k} \right)_{k \in \N}\right\|_1}
\\
\hfill\text{since $|w|>1$;}
\\
\shoveleft{
\le \left\| \left(\left[ \prod_{j=1}^{n} w^{-1} \right] x_{k} \right)_{k \in \N}\right\|_1}=|w|^{-n}\|x\|_1.
\hfill
\end{multline*}

In view of $\|x\|_1<\infty$ being fixed, $\left(A_w^nx\right)_{n\in \N}$ eventually zero, and $|w|>1$, 
\begin{equation*}\label{UWBS1}
\begin{aligned}
\forall\, x\in c_{00}\ &\exists\, \alpha(x)\in \left(|w|^{-1},1\right)\subset (0,1), \ \exists\, c=c(x,\alpha)>\|x\|_1\ge 0 \ \forall\, n\in \N:\\
&\max\left(\left\|A_w^nx\right\|, \left\|B_w^nx\right\|\right)\le \max\left(\left\|A_w^nx\right\|, |w|^{-n}\|x\|_1\right)\le c\alpha^n.
\end{aligned}
\end{equation*}

i.e. the second condition of the \textit{Sufficient Condition for Linear Chaos} (Theorem \ref{SCC}) is also met.

By Lemma \ref{lem}, the \textit{Sufficient Condition for Linear Chaos} (Theorem \ref{SCC}), and the \textit{Chaoticity of Powers Corollary} (Corollary \ref{CP}), we conclude that the operator $A_w$ is \textit{chaotic} as well as its every power $A_w^n$ $(n\in \N)$. 

What follows almost verbatim mimics the reasoning of \cite[Theorem $3.1$]{arXiv:1811.06640} (also \cite[Theorem $3.2 $]{arXiv:2203.02032}).

For arbitrary $\lambda \in \F$ ($\F:=\R$ or $\F:=\C$) and $x:=(x_k)_{\N} \in D(A_w)$, the equation
\begin{equation}\label{ev}
A_wx=\lambda x
\end{equation}
is equivalent to
\[
(w^k x_{k+1})_{k\in \N}=\lambda(x_k)_{k\in \N},
\]
i.e.,
\[
w^k x_{k+1}=\lambda x_k,\ k\in \N
\]

Whence, we recursively infer that
\[
x_k=\left[\prod_{j=1}^{k-1}\frac{\lambda}{w^{k-j}}\right]x_1
=\dfrac{\lambda^{k-1}}{w^{\sum_{j=1}^{k-1}(k-j)}}x_1
=\dfrac{\lambda^{k-1}}{w^{\frac{k(k-1)}{2}}}x_1
=\left(\dfrac{\lambda}{w^{\frac{k}{2}}}\right)^{k-1}x_1,\
k\in \N,
\]
where for $\lambda=0$, $0^0:=1$.

Considering that $|w|>1$, for all sufficiently large $k\in \N$, we have:
\[
\left|\dfrac{\lambda}{w^{\frac{k}{2}}}\right|^{k-1}
=\left(\dfrac{|\lambda|}{|w|^{\frac{k}{2}}}\right)^{k-1}
\le \left(\dfrac{1}{2}\right)^{k-1},
\]
which implies, by the \textit{Comparison Test} and the \emph{Continuous and Dense Embeddings Proposition} (Proposition \ref{CDE}),
\[
y:=(y_k)_{k\in \N}:=\left(\left(\dfrac{\lambda}{w^{\frac{k}{2}}}\right)^{k-1}\right)_{k\in \N}\in l_1\subset X.
\]

Further, since
\[
w^ky_{k+1}=w^k\dfrac{\lambda^{k}}{w^{\sum_{j=1}^{k}(k+1-j)}}
=\dfrac{\lambda^{k}}{w^{\sum_{j=2}^{k}(k+1-j)}}=\dfrac{\lambda^{k}}{w^{\frac{k(k-1)}{2}}}x_1
=\left(\dfrac{\lambda}{w^{\frac{k-1}{2}}}\right)^{k},\ k\in \N,
\]
we similarly conclude that
\[
(w^ky_{k+1})_{k\in \N}\in X,
\]
and hence,
\[
y\in D(A_w)\setminus \{0\}.
\]

Thus, we have shown that, for any $\lambda \in \F$, 
\[
\ker(A_w-\lambda I)=\spa\left(\left\{y\right\}\right)\subseteq D(A_w),
\]
and hence,
\[
\dim\ker(A_w-\lambda I)=1,
\]
which completes the proof.
\end{proof}

\section{More Hypercyclicity and Linear Chaos}

Here, we discuss how the known chaos generates new chaos via the conjugacy relative the isometric isomorphism between the spaces $(X,\|\cdot\|)$ and $l_1$ (see Section \ref{secii}).

\begin{thm}[More Bounded Linear Chaos in $X$]\ \\
For $w\in \F$, the bounded linear operator
\[
X\ni x:=(x_k)_{k\in \N}\mapsto \hat{A}_wx:=w\left(\frac{k}{k+1}x_{k+1}\right)_{k\in \N}\in X
\]
on the space $(X,\|\cdot\|)$ is chaotic as well as its every power 
\[
X\ni x:=(x_k)_{k\in \N}\mapsto {\hat{A}_w}^nx=w^n\left(\frac{k}{k+n}x_{k+n}\right)_{k\in \N}\in X, \ n\in \N
\]
when $|w|>1$ and nonhypercyclic otherwise when $|w|\le 1$.

Provided the space $(X,\|\cdot\|)$ is complex
\[
\sigma\left(\hat{A}_w\right)=\left\{ \lambda\in \C \,\middle|\, |\lambda|\le |w| \right\}
\]
with
\[
\sigma_p\left(\hat{A}_w\right)=\left\{ \lambda\in \C \,\middle|\, |\lambda|<|w| \right\}\quad \text{and}\quad
\sigma_c\left(\hat{A}_w\right)=\left\{ \lambda\in \C \,\middle|\, |\lambda|=|w| \right\}.
\]
\end{thm}

\begin{proof}
Let $w\in \F$ be arbitrary. 

The weighed backward shift
\[
l_1\ni z:=(z_k)_{k\in \N}\mapsto A_wz:=w(z_{k+1})_{k\in \N}\in l_1
\]
is a \textit{bounded linear operator} which is \emph{chaotic} \cite{Rolewicz} as well as its every power $\hat{A}_w^n$ ($n\in \N$) when $|w|>1$, being subject to the \textit{Sufficient Condition for Linear Chaos} (Theorem \ref{SCC}) (see \cite[Examples $3.1$]{arXiv:2106.14872}), and is \emph{nonhypercyclic} when $|w|\le 1$ since $\|A_w\|=|w|\le 1$ (see, e.g., \cite{Markin2020EOT}).

By the \emph{Isometric Isomorphisms Proposition} (Proposition \ref{IIP}), the mapping
\begin{equation}\label{ii1}
X\ni x:=(x_k)_{k\in \N}\mapsto Jx:=\left(\frac{x_{k+1}}{k+1}-\frac{x_k}{k}\right)_{k\in \N}\in l_1
\end{equation}
is an \textit{isometric isomorphism} between $(X,\|\cdot\|)$ and $l_1$ with the inverse 
\begin{equation}\label{ii2}
l_1\ni z:=(z_k)_{k\in \N}\mapsto J^{-1}z:=\left(k\sum_{j=0}^{k-1}z_j\right)_{k\in \N}\in X,
\end{equation}
where 
\[
z_0:=-\sum_{k=1}^\infty z_k.
\]

On the space $(X,\|\cdot\|)$, consider the linear operator
\[
\hat{A}_w:=J^{-1}A_wJ
\]
naturally emerging from the commutative diagram
\begin{equation*}
\begin{tikzcd}[sep=large]
l_1\arrow[r, "A_w"] & l_1 \\
X\arrow[u, "J"] \arrow[r, "\hat{A}_w"]& X\arrow[u, "J"]
\end{tikzcd}.
\end{equation*}

Since
\[
\hat{A}_w^n:=J^{-1}A_w^nJ,\ n\ \in \N,
\]
where
\[
l_1\ni z:=(z_k)_{k\in \N}\mapsto A_w^nz:=w^n(z_{k+n})_{k\in \N}\in l_1,
\]
for any $x:=(x_k)_{k\in \N}\in X$,
\begin{align*}
\hat{A}_w^nx&=J^{-1}A_w^nJx=J^{-1}A_w^n\left(\frac{x_{k+1}}{k+1}-\frac{x_k}{k}\right)_{k\in \N}=J^{-1}w^n\left(\frac{x_{k+1+n}}{k+1+n}-\frac{x_{k+n}}{k+n}\right)_{k\in \N}\\
&=w^n\left(k\sum_{j=0}^{k-1}z_j \right)_{k\in \N}
=w^n\left(\frac{k}{k+n}x_{k+n}\right)_{k\in \N}
\end{align*}
since, for
\[
z_k:=\frac{x_{k+1+n}}{k+1+n}-\frac{x_{k+n}}{k+n},\ k\in \N,
\]
\[
z_0:=-\sum_{k=1}^\infty z_k=-\sum_{k=1}^\infty\left(\frac{x_{k+1+n}}{k+1+n}-\frac{x_{k+n}}{k+n}\right)
=\frac{x_{1+n}}{1+n},
\]
and hence,
\[
\sum_{j=0}^{k-1}z_j =z_0+\sum_{j=1}^{k-1}\left(\frac{x_{j+1+n}}{j+1+n}-\frac{x_{j+n}}{j+n}\right)
=\frac{x_{1+n}}{1+n}+\frac{x_{k+n}}{k+n}-\frac{x_{1+n}}{1+n}=\frac{x_{k+n}}{k+n},\ k\in \N.
\]

Further, since $J:X\to l_1$ is an \emph{isometric isomorphism}, the operator $\hat{A}_w^n$ ($n\in \N$) inherits the \textit{boundedness} and \textit{chaotic/hypercyclic properties} of $A_w$ as well as its \textit{spectral structure}. 

Hence, $\hat{A}_w$ is \emph{chaotic} as well as its every power $\hat{A}_w^n$ ($n\in \N$) when $|w|>1$ and is \emph{nonhypercyclic} when $|w|\le 1$.

Provided the underlying space is complex, the spectral part of the theorem follows from the
fact that 
\[
\sigma\left(A_w\right)=\left\{ \lambda\in \C \,\middle|\, |\lambda|\le |w| \right\}
\]
with
\[
\sigma_p\left(A_w\right)=\left\{ \lambda\in \C \,\middle|\, |\lambda|<|w| \right\}\quad \text{and}\quad
\sigma_c\left(A_w\right)=\left\{ \lambda\in \C \,\middle|\, |\lambda|=|w| \right\}
\]
(see, e.g., \cite{Markin2020EOT}).
\end{proof}

\begin{thm}[More Unbounded Linear Chaos in $X$]\ \\
For arbitrary $w\in \F$ with $|w|>1$, the linear operator
\[
\hat{A}_wx:=\left(k\sum_{j=k}^\infty w^j\left(\frac{x_{j+1}}{j+1}-\frac{x_{j+2}}{j+2}\right)\right)_{k\in \N}
\]

in the space $(X,\|\cdot\|)$ with domain
\[
D\left(\hat{A}_w\right):=\left\{x:=(x_k)_{k\in \N}\in X\, \middle|\, \left(w^k\left(\frac{x_{k+2}}{k+2}-\frac{x_{k+1}}{k+1}\right)\right)_{k\in \N}\in l_1\right\}
\]

is unbounded and chaotic as well as its every power
\[
\hat{A}_w^nx=\left(k\sum_{j=k}^\infty\left[\prod_{m=j}^{j+n-1}w^m\right]\left(\frac{x_{j+n}}{j+n}-\frac{x_{j+1+n}}{j+1+n}\right)\right)_{k\in \N}
\]


in the space $(X,\|\cdot\|)$ with domain
\[
D\left({\hat{A}_w}^n\right)=\left\{x:=(x_k)_{k\in \N}\in X\, \middle|\, \left(\left[\prod_{j=k}^{k+n-1}w^j\right]\left(\frac{x_{k+1+n}}{k+1+n}-\frac{x_{k+n}}{k+n}\right)\right)_{k\in \N}\in l_1 \right\}.
\]

Furthermore, each $\lambda \in \F$ is an eigenvalue for $\hat{A}_w$ of geometric multiplicity $1$, i.e.,
\[
\dim\ker(\hat{A}_w-\lambda I)=1.
\]
In particular, provided the space $(X,\|\cdot\|)$ is complex,
\[
\sigma_p\left(\hat{A}_w\right)=\C.
\]
\end{thm}

\begin{proof}
Let $w\in \F$ with $|w|>1$ be arbitrary.

As is known, the unbounded weighted backward shift
\[
D(A_w)\ni z:=(z_k)_{k\in \N}\mapsto A_wz:=\left(w^kz_{k+1}\right)_{k\in \N}\in l_1
\]
in $l_1$ with domain
\[
D(A_w):=\left\{z:=(z_k)_{k\in \N}\in l_1\,\middle |\, \left(w^kz_{k+1}\right)_{k\in \N}\in l_1\right\}.
\]
is a \emph{chaotic} linear operator \cite[Theorem $3.1$]{arXiv:1811.06640} along with every power $A_w^n$ ($n\in \N$) \cite[Examples $3.1$]{arXiv:2106.14872}.

Based the \emph{Isometric Isomorphisms Proposition} (Proposition \ref{IIP}), for the \textit{isometric isomorphism} $J:X\to l_1$ subject to \eqref{ii1} and \eqref{ii2}, in the space $(X,\|\cdot\|)$, consider the linear operator
\[
\hat{A}_w:=J^{-1}A_wJ
\]
as is given by the following commutative diagram:
\begin{equation*}
\begin{tikzcd}[sep=large]
l_1\supseteq D(A_w)\arrow[r, "A_w"] & l_1 \\
X\supseteq D\left(\hat{A}_w\right)\arrow[u, "J"] \arrow[r, "\hat{A}_w"]& X\arrow[u, "J"]
\end{tikzcd},
\end{equation*}
where the domain of $\hat{A}_w$ is
\[
D\left(\hat{A}_w\right):=J^{-1}\left(D(A_w)\right).
\]

Since
\[
{\hat{A}_w}^n=J^{-1}A_w^nJ,\ n\in \N,
\]

where
\[
A_w^nz=\left(\left[\prod_{j=k}^{k+n-1}w^j\right]z_{k+n}\right)_{k\in \N}
\]
with
\[
D\left(A_w^n\right)=\left\{z:=(z_k)_{k\in \N}\in l_1\, \middle|\, \left(\left[\prod_{j=k}^{k+n-1}w^j\right]z_{k+n}\right)_{k\in \N}\in l_1\right\}
\]
(cf. \cite{arXiv:1811.06640,arXiv:2203.02032}), in view of \eqref{ii1},
\[
D\left({\hat{A}_w}^n\right)=\left\{x:=(x_k)_{k\in \N}\in X\, \middle|\, \left(\left[\prod_{j=k}^{k+n-1}w^j\right]\left(\frac{x_{k+1+n}}{k+1+n}-\frac{x_{k+n}}{k+n}\right)\right)_{k\in \N}\in l_1 \right\}.
\]

Moreover, for arbitrary $x:=(x_k)_{k\in \N}\in X$, in view of \eqref{ii1} and \eqref{ii2},
\begin{align*}
{\hat{A}_w}^nx&=J^{-1}A_w^nJx=J^{-1}A_w^n\left(\frac{x_{k+1}}{k+1}-\frac{x_k}{k}\right)_{k\in \N}\\
&=J^{-1} \left(\left[\prod_{j=k}^{k+n-1}w^j\right]\left(\frac{x_{k+1+n}}{k+1+n}-\frac{x_{k+n}}{k+n}\right)\right)_{k\in \N}\\
&=\left(k\left[z_0+\sum_{j=1}^{k-1}\left[\prod_{m=j}^{j+n-1}w^m\right]\left(\frac{x_{j+1+n}}{j+1+n}-\frac{x_{j+n}}{j+n}\right)\right]\right)_{k\in \N}\\
&=\left(-k\sum_{j=k}^{\infty}\left[\prod_{m=j}^{j+n-1}w^m\right]\left(\frac{x_{j+1+n}}{j+1+n}-\frac{x_{j+n}}{j+n}\right)\right)_{k\in \N}\\
&=\left(k\sum_{j=k}^{\infty}\left[\prod_{m=j}^{j+n-1}w^m\right]\left(\frac{x_{j+n}}{j+n}-\frac{x_{j+1+n}}{j+1+n}\right)\right)_{k\in \N}
\end{align*}

where $z_0:=-\displaystyle\sum_{k=1}^\infty\left[\prod_{j=k}^{k+n-1}w^j\right]\left(\frac{x_{k+1+n}}{k+1+n}-\frac{x_{k+n}}{k+n}\right)$.

Further, since $J:X\to l_1$ is an \emph{isometric isomorphism}, the operator ${\hat{A}_w}^n$ ($n\in \N$) inherits the \emph{unboundedness} and \emph{chaoticity} of $A_w$ as well as its \emph{eigenvalues} coupled with their \emph{geometric multiplicities}. 

Therefore, the operator $\hat{A}_w$ is \emph{unbounded} and  \emph{chaotic} as well as its every power ${\hat{A}_w}^n$ ($n\in \N$).

Furthermore, in view of,
\[
\forall\, \lambda\in \F:\ \dim\ker(A_w-\lambda I)=1
\]
(see \cite[Theorem $3.1$]{arXiv:1811.06640}),  each $\lambda\in \F$ is a \emph{simple eigenvalue} for $\hat{A}_w$.
\end{proof}

\begin{thm}[More Bounded Linear Chaos in $l_1$]\ \\
For $w\in \F$, the bounded linear operator
\[
l_1\ni z:=(z_k)_{k\in \N}\mapsto \hat{A}_wz:=w\left(\frac{k+2}{k+1}z_{k+1}-\frac{1}{k(k+1)}\sum_{j=0}^{k}z_j\right)_{k\in \N}\in l_1 
\]
with 
\[
z_0:=-\sum_{k=1}^\infty z_k,
\]
on the space $l_1$ is 
\begin{enumerate}[label=\arabic*.]
\item non-hypercyclic for $|w|<1$,
\item hypercyclic but not chaotic for $|w|=1$, and
\item chaotic as well as its every power
\[
l_1\ni z:=(z_k)_{k\in \N}\mapsto {\hat{A}_w}^n=w^n\left(\frac{k+1+n}{k+1}z_{k+n}-\frac{n}{k(k+1)}\sum_{j=0}^{k-1+n}z_j\right)_{k\in \N}\in l_1
\]
with 
\[
z_0:=-\sum_{k=1}^\infty z_k,
\]
for $|w|>1$.
\end{enumerate}

Provided the space $l_1$ is complex (i.e., $\F=\C$),
\[
\sigma\left(\hat{A}_w\right)=\left\{ \lambda\in \C \,\middle|\, |\lambda|\le |w| \right\}
\]
with
\[
\sigma_p(\hat{A}_w)=\left\{ \lambda\in \C \,\middle|\, |\lambda|< |w| \right\}\cup \left\{w\right\}
\quad
\text{and}
\quad
\sigma_c(\hat{A}_w)=\left\{ \lambda\in \C \,\middle|\, |\lambda|=|w| \right\}\setminus \left\{w\right\}.
\]
\end{thm}

\begin{proof}
For $w\in \F$, let
\[
X\ni x:=(x_k)_{k\in \N}\mapsto A_wx:=w(x_{k+1})_{k\in \N}\in X.
\]

By the \emph{Bounded Weighted Backward Shifts Theorem} (Theorem \ref{BWBS}), $A_w$ is a \emph{bounded} linear operator that is
\begin{enumerate}
\item[(i)] \emph{nonhypercyclic} for $|w|<1$,
\item[(ii)] \emph{hypercyclic} but \emph{not chaotic} for $|w|=1$, and
\item[(iii)] \emph{chaotic} along with every power $A_w^n$, $n\in \N$, for $|w|>1$.
\end{enumerate}

By the \emph{Isometric Isomorphisms Proposition} (Proposition \ref{IIP}), in view of $\eqref{ii1}$ and $\eqref{ii2}$, on the space $l_1$, consider the linear operator
\[
\hat{A}_w:=JA_wJ^{-1}.
\]

i.e., the following diagram commutes.
\begin{equation*}
\begin{tikzcd}[sep=large]
l_1\arrow[r, "\hat{A}_w"] & l_1 \\
X\arrow[u, "J"] \arrow[r, "A_w"]& X\arrow[u, "J"]
\end{tikzcd}.
\end{equation*}

Since 
\[
{\hat{A}_w}^n=JA_w^nJ^{-1},\ n\in \N,
\]
where
\[
X\ni x:=(x_k)_{k\in \N}\mapsto A_w^nx=w^n\left(x_{k+n}\right)_{k\in \N}\in X,
\]

for any $z:=(z_k)_{k\in \N}\in l_1$, in view of \eqref{ii1} and \eqref{ii2}, the following holds:
\begin{align*}
{\hat{A}_w}^nz&=JA_w^nJ^{-1}z=JA_w^n\left(k\sum_{j=0}^{k-1}z_j\right)_{k\in \N}=Jw^n\left((k+n)\sum_{j=0}^{k-1+n}z_j\right)_{k\in \N}\\
&=w^n\left(\frac{k+1+n}{k+1}\sum_{j=0}^{k+n}z_j-\frac{k+n}{k}\sum_{j=0}^{k-1+n}z_j\right)_{k\in \N}\\
&=w^n\left(\frac{k+1+n}{k+1}z_{k+n}+\frac{k+1+n}{k+1}\sum_{j=0}^{k-1+n}z_j-\frac{k+n}{k}\sum_{j=0}^{k-1+n}z_j\right)_{k\in \N}\\
&=w^n\left(\frac{k+1+n}{k+1}z_{k+n}-\frac{n}{k(k+1)}\sum_{j=0}^{k-1+n}z_j\right)_{k\in \N}
\end{align*}
where 
\[
z_0:=-\sum_{k=1}^\infty z_k.
\]

Furthermore, since $J$ is an \emph{isometric isomorphism}, $\hat{A}_w$ inherits the \emph{boundedness} and \emph{chaotic/hypercyclic properties} of $A_w$ as well as its \emph{spectral structure}.

Hence, $\hat{A}_w$ is a bounded linear operator that is
\begin{enumerate}[label=\arabic*.]
\item \emph{nonhypercyclic} for $|w|<1$,
\item \emph{hypercyclic} but \emph{not chaotic} for $|w|=1$, and
\item \emph{chaotic} along with every power $A_w^n$, $n\in \N$, for $|w|>1$.
\end{enumerate}

Provided the underlying space is complex, the spectral part of this theorem follows from the \emph{Bounded Weighted Backward Shifts Theorem} (Theorem \ref{BWBS}) where
\[
\sigma(A_w)=\left\{ \lambda\in \C \,\middle|\, |\lambda|\le |w| \right\}
\]
with
\[
\sigma_p(A_w)=\left\{ \lambda\in \C \,\middle|\, |\lambda|< |w| \right\}\cup \left\{w\right\}
\quad
\text{and}
\quad
\sigma_c(A_w)=\left\{ \lambda\in \C \,\middle|\, |\lambda|=|w| \right\}\setminus \left\{w\right\}.
\]
\end{proof}

\begin{thm}[More Unbounded Linear Chaos in $l_1$]\ \\
For an arbitrary $w\in \F$ with $|w|>1$, the linear operator
\[
\hat{A}_wz:=\left(w^{k+1}\frac{k+2}{k+1}\sum_{j=0}^{k+1}z_j-w^k\frac{k+1}{k}\sum_{j=0}^{k}z_j\right)_{k\in \N}
\]
with
\[
z_0:=-\sum_{k=1}^\infty z_k,
\] 
and domain
\[
D\left(\hat{A}_w\right):=\left\{z:=(z_k)_{k\in \N}\in l_1\, \middle|\, \left(w^k(k+1)\sum_{j=0}^{k}z_j\right)_{k\in \N}\in X \right\}
\]
is unbounded and chaotic as well as its every power
\[
{\hat{A}_w}^nz=\left(\left[\displaystyle\prod_{m=k}^{k-1+n}w^m\right]\left(w^n\frac{k+1+n}{k+1}\sum_{j=0}^{k+n}z_j-\frac{k+n}{k}\sum_{j=0}^{k-1+n}z_j\right)\right)_{k\in \N}, \ n\in \N,
\]
with 
\[
z_0:=-\sum_{k=1}^\infty z_k,
\] 
and domain
\[
D\left({\hat{A}_w}^n\right)=\left\{z:=(z_k)_{k\in \N}\in l_1\, \middle|\, \left(\left[\prod_{m=k}^{k+n-1}w^m\right](k+n)\sum_{j=0}^{k-1+n}z_j\right)_{k\in \N}\in X \right\}.
\]

Furthermore, each $\lambda \in \F$ is an eigenvalue for $\hat{A}_w$ of geometric multiplicity $1$, i.e.,
\[
\dim\ker(\hat{A}_w-\lambda I)=1.
\]
In particular, provided the space $(X,\|\cdot\|)$ is complex,
\[
\sigma_p\left(\hat{A}_w\right)=\C.
\]
\end{thm}

\begin{proof}
For arbitrary $w\in \F$ with $|w|>1$, let
\[
D(A_w)\ni x:=(x_k)_{k\in \N}\mapsto A_wx:=\left(w^kx_{k+1}\right)_{k\in \N}\in X
\]

where
\[
D(A_w):=\left\{x:=(x_k)_{k\in \N}\in X\,\middle |\, \left(w^kx_{k+1}\right)_{k\in \N}\in X\right\}.
\]

By Lemma \ref{lem} and the \emph{Unbounded Weighted Backward Shifts Theorem} (Theorem \ref{UWBS}), $A_w$ is \emph{unbounded} and \emph{chaotic} along with every power $A_w^n$ ($n\in \N$).

By the \emph{Isometric Isomorphisms Proposition} (Proposition \ref{IIP}), in view of \eqref{ii1} and \eqref{ii2}, in the space $l_1$, consider the following linear operator
\[
\hat{A}_w:=JA_wJ^{-1}
\]

that emerges from the commutative diagram below

\begin{equation*}
\begin{tikzcd}[sep=large]
l_1\supseteq D\left(\hat{A}_w\right)\arrow[r, "\hat{A}_w"] & l_1 \\
X\supseteq D(A_w)\arrow[u, "J"] \arrow[r, "A_w"]& X\arrow[u, "J"]
\end{tikzcd}.
\end{equation*}

where the domain of $\hat{A}_w$ is
\[
D\left(\hat{A}_w\right):=J\left(D(A_w)\right).
\]

Since
\[
{\hat{A}_w}^n=JA_w^nJ^{-1},\ n\in \N,
\]
where, by Lemma \ref{lem},
\[
A_w^{n}x= \left( \left[ \prod_{j=k}^{k+n-1} w^j \right]x_{k+n} \right)_{k \in \N},
\]
with domain
\[
D(A_w^{n})= \left\{ x := \left(x_k \right)_{k \in \N} \in X \,\middle|\, \left( \left[ \prod_{j=k}^{k+n-1} w^j \right]x_{k+n} \right)_{k \in \N} \in X \right\},
\]
in view of \eqref{ii2},
\[
D\left({\hat{A}_w}^n\right)=\left\{z:=(z_k)_{k\in \N}\in l_1\, \middle|\, \left(\left[\prod_{m=k}^{k+n-1}w^m\right](k+n)\sum_{j=0}^{k-1+n}z_j\right)_{k\in \N}\in X \right\}
\]
with 
\[
z_0:=-\sum_{k=1}^\infty z_k.
\] 

Moreover, for arbitrary $z:=(z_k)_{k\in \N}\in l_1$, in view of \eqref{ii1} and \eqref{ii2},
\begin{align*}
{\hat{A}_w}^nz&=JA_w^nJ^{-1}z=JA_w^n\left(k\sum_{j=0}^{k-1}z_j\right)_{k\in \N}=J\left(\left[\prod_{m=k}^{k+n-1}w^m\right](k+n)\sum_{j=0}^{k-1+n}z_j\right)_{k\in \N}\\
&=\left(\frac{(k+1+n)\left[\displaystyle\prod_{m=k+1}^{k+n}w^m\right]\displaystyle\sum_{j=0}^{k+n}z_j}{k+1}-\frac{(k+n)\left[\displaystyle\prod_{m=k}^{k-1+n}w^m\right]\displaystyle\sum_{j=0}^{k-1+n}z_j}{k}\right)_{k\in \N}\\
&=\left(\left[\displaystyle\prod_{m=k}^{k-1+n}w^m\right]\left(w^n\frac{k+n+1}{k+1}\sum_{j=0}^{k+n}z_j-\frac{k+n}{k}\sum_{j=0}^{k-1+n}z_j\right)\right)_{k\in \N}
\end{align*}
where 
\[
z_0:=-\sum_{k=1}^\infty z_k.
\]  

Further, since $J:X\to l_1$ is an \emph{isometric isomorphism}, the operator ${\hat{A}_w}^n$ ($n\in \N$) inherits the \emph{boundedness} and \emph{chaoticity} of $A_w$ as well as its \emph{eigenvalues} coupled with their \emph{geometric multiplicities}. 

Therefore, $\hat{A}_w$ is \emph{unbounded} and  \emph{chaotic} as well as its every power ${\hat{A}_w}^n$ ($n\in \N$).

Furthermore, by the \emph{Unbounded Weighted Backward Shifts Theorem} (Theorem \ref{UWBS}), every $\lambda\in \F$ is an \emph{simple eigenvalue} for $A_w$.
\end{proof}

 
\end{document}